\documentclass[review,3p,10pt]{elsarticle}
\biboptions{sort&compress}

\usepackage{amssymb}
\usepackage{amsmath}
\usepackage{cases}
\usepackage{epstopdf}
\usepackage{fixltx2e}
\usepackage{mathtools}
\usepackage{array}
\usepackage{bm}
\usepackage{multirow}
\usepackage{makecell}
\usepackage{tabularx}
\usepackage{longtable}
\usepackage{float}
\usepackage{caption}
\usepackage{color}
\usepackage{algorithm}
\usepackage{algorithmic}
\usepackage{amsthm}
\usepackage{mathrsfs}
\usepackage{amssymb}

\newtheorem{proposition}{Proposition}
\newtheorem{remark}{Remark}
\usepackage{lineno}
\usepackage{booktabs}
\usepackage{threeparttable} 
\usepackage{caption}
\usepackage{graphicx}
\usepackage{booktabs}
\usepackage{multirow}
\usepackage{threeparttable}
\usepackage{adjustbox}
\usepackage{float}
\usepackage{subcaption}
 \usepackage{pdflscape}
 \usepackage{adjustbox}
 \usepackage{multicol}
\usepackage{enumitem}
\usepackage[table,xcdraw]{xcolor}

\usepackage{url}

\newcommand{\tabincell}[2]{\renewcommand\arraystretch{0.8}\begin{tabular}{@{}#1@{}}#2\end{tabular}}

\journal{arXiv.org}

\begin{document}
\begin{frontmatter}


\title{Frequency Security-Aware Production Scheduling of Utility-Scale Off-Grid Renewable P2H Systems Coordinating Heterogeneous Electrolyzers}

\author[label1,label2]{Jie~Zhu}
\author[label1]{Yiwei~Qiu\corref{cor1}}
\ead{ywqiu@scu.edu.cn}
\author[label1]{Yangjun~Zeng}
\author[label2]{Shahab~Dehghan}
\author[label2]{Sheng~Wang}
\author[label1]{Shi~Chen}
\author[label1]{Buxiang~Zhou}

\address[label1]{College of Electrical Engineering, Sichuan University, Chengdu, 610065, China}
\address[label2]{School of Engineering, Newcastle University, Newcastle upon Tyne, NE1 7RU, UK}
\cortext[cor1]{Corresponding author}

\begin{abstract}
  Renewable power-to-hydrogen (ReP2H) enables large-scale renewable energy utilization and supports the decarbonization of hard-to-abate sectors, such as chemicals and maritime transport, via hydrogen-based renewable ammonia and methanol fuels. As a result, utility-scale ReP2H projects are expanding worldwide. However, off-grid ReP2H systems exhibit low inertia due to their converter-dominated nature, making frequency security a critical concern. Although recent studies show that electrolyzers can contribute to frequency regulation (FR), their support capability depends on operating states and loading levels, creating a trade-off between hydrogen output and frequency security. To address this challenge, this work develops a unified co-optimization framework for frequency security-aware production scheduling of utility-scale off-grid ReP2H systems coordinating heterogeneous electrolyzers.  A system-level frequency response model is established to capture multi-stage FR from alkaline water electrolyzers (AWEs), proton exchange membrane electrolyzers (PEMELs), and other resources, including ammonia-fueled generators retrofitted in co-located chemical plants, battery energy storage, and wind turbines (WTs). Stage-wise transient frequency security constraints are derived, reformulated into tractable forms, and embedded into production scheduling, enabling coordinated on/off switching and load allocation across electrolyzers to maximize hydrogen output under uncertain renewable power input while enforcing frequency security constraints. Case studies based on real-world systems demonstrate that the proposed approach allows HPs to replace 55.52\% and 96.85\% of FR reserves from WTs and AFGs, respectively, while maintaining comparable hydrogen output. Year-long simulations show an average 28.96\% increase in annual net profit resulting from reduced reliance on conventional reserves.
\end{abstract}


\end{frontmatter}

\break

\section*{\footnotesize Nomenclature}
\addcontentsline{toc}{section}{Nomenclature}

\begin{multicols}{2}
	
	\footnotesize
	\setlength{\columnsep}{18pt} 
	
	\noindent\textbf{Abbreviations}
	\begin{description}[style=standard,leftmargin=!,labelwidth=1.8cm,labelsep=0.5em,itemsep=0pt,topsep=0pt, font=\normalfont]
		\item[AFG] Ammonia-fueled generator
		\item[AWE] Alkaline water electrolyzers
		\item[BES] Battery energy storage
        \item[COI] Center-of-inertia
		\item[DRCC] Distributionally robust chance constraints
		\item[EDL] Electrical double-layer
		\item[EL] Electrolyzer
		\item[FC] Frequency constrained
		\item[FSPS]  Frequency-supporting production scheduling
		\item[GFM] Grid-forming
		\item[HP]  Hydrogen plant
		\item[PEMEL] Proton exchange membrane electrolyzers
		\item[PV]  Photovoltaic
		\item[PFR] Primary frequency regulation
		\item[ReP2H] Renewable power-to-hydrogen
		\item[RoCoF] Rate of change of frequency
		\item[VI] Virtual inertia
		\item[WT]  Wind turbine
	\end{description}
	
	\vspace{5pt}
	\noindent\textbf{Indices}
	\begin{description}[style=standard,leftmargin=!,labelwidth=1.8cm,labelsep=0.5em,itemsep=0pt,topsep=0pt]
		\addcontentsline{toc}{section}{Nomenclature}
		\item[$t$]
		Index of scheduling time interval
		
		\item[$b/e/g/w$]
		Index for BES/EL/AFG/WT
		
		\item[$i/j/ij $]
		Index for bus and branch between bus $i$ and $j$
		
		\item[$\mathcal{F}(i) $]
		Set of parents of bus $i$ in the DistFlow model
		
		\item[$\varsigma(i)$]
		Set of children of bus $i$ in the DistFlow model
		
	\end{description}
	
	\vspace{5pt}
	\noindent\textbf{Variables}
	\begin{description}[style=standard,leftmargin=!,labelwidth=1.8cm,labelsep=0.5em,itemsep=0pt,topsep=0pt]
		\addcontentsline{toc}{section}{Nomenclature}
		\item[$\Delta f$]
		Center of inertia frequency deviation
		
		\item[$\Delta f_{b/e/g} $]
		Frequency deviation at the BES/HP/AFG bus
		
		\item[$ h_{e}^\mathrm{cool}$]
		Cooling heat of EL
		
		\item[$I_e$]
		Electrolysis current of the EL
		
		
		%
		
		\item[$P_e^\mathrm{auxi/sb}$]
		Auxiliary device/standby power of EL
		
		\item[$P_c^\mathrm{comp}$]
		Hydrogen compressor power
		
		\item[$P_e^\mathrm{cool}$]
		Cooler power of EL
		
		\item[$P_e^\mathrm{pump}$]
		Pump power of EL
		
		
		\item[$P_{b/g/w/s}$]
		Power of BES/AFG/WT/PV plant
		
		\item[$P_{e/d}$]
		Power of EL/downsrtream chemical industry load
		
		\item[$P_{b,t}^{\mathrm{C/D}}$]
		Charging/discharging power of BES
		
		\item[$\tilde{P}_w/\hat{P}_{w/s} $]
		Available/forecast power of WT/PV plant
		
		\item[$\Delta P_{g/w/b/e}^{\mathrm{PFR}} $]
		PFR power
		
		\item[$\Delta P_{{\mathrm{dis}}} $]
		Disturbance power
		
		\item[$q_{e}^{\mathrm{H_2}}$]
		Hydrogen production flow rate
		
		\item[$ R_{g/w/b/e}^{\mathrm{PFR}} $]
		PFR reserve
		
		\item[$ R_{g/w/b/e}^{\mathrm{+/-}} $]
		Up/down regulation reserve
		
		\item[$E_{b}$]
		State of charge
		
		\item[$ T_e$]
		Operation temperature of EL
		
		\item[$T_{e}^\mathrm{cool}$]
		EL cooling system temperature
		
		\item[$x_{b}^{\mathrm{C}}$]
		Charging state of BES
		
		\item[$x_{e}^{\mathrm{st/sb/sp}}$]
		Start-up/standby/shutdown state of EL
		
		\item[$x_{g}^{\mathrm{st}}$]
		Start-up states of AFG
		
		\item[$z_{e}^{\mathrm{up,c/up,h}}$]
		Hot/cold start-up operation of EL
		
		\item[$z_{g}^{\mathrm{up/down}}$]
		Start-up/shutdown operation of AFG
		
		\item[$z_{e}^{\mathrm{down}}$]
		Shutdown operation of EL
		
		\item[$\xi_{w/s} $]
		Forecast error of WT/PV plant
		
	\end{description}
	
	\vspace{5pt}
	\noindent\textbf{Parameters}
	\begin{description}[style=standard,leftmargin=!,labelwidth=1.8cm,labelsep=0.5em,itemsep=0pt,topsep=0pt]
		\addcontentsline{toc}{section}{Nomenclature}
		
		\item[$a_{e}^\mathrm{cool}$]
		Swatch factor of temperature to heat
		
		\item[$c_{e}^{\mathrm{H_2}}$]
		Green hydrogen price
		
		\item[$c_{g}^{\mathrm{NH_3}}$]
		Generation cost of AFG
		
		\item[$c_{e/g}^{\mathrm{up/down}}$]
		Start-up/shutdown cost of EL/AFG
		%
		
%
%
		
		\item[$D_d$]
		Damping rate of frequency-dependent load
		
		\item[$D_{g/b}$]
		Damping/virtual damping from AFG/BES
		
		\item[$ \Delta f_{\mathrm{db1/db2}}$]
		Frequency deadband 1/2
		
		\item[$ \Delta f_{\mathrm{nadir}}^{\mathrm{lim}}$]
		Limit of frequency deviation nadir
		
		\item[$ \Delta f_{\mathrm{qss}}^{\mathrm{lim}}$]
		Limit of quasi-steady-state frequency deviation
		
		\item[$ \Delta f_{\mathrm{RoCoF}}^{\mathrm{lim}}$]
		Limit of rate of change of frequency
		
		\item[$H_{g} $]
		Inertia constant of AFG
		
		\item[$H_{b/e} $]
		Virtual inertia from BES/EL
		
		\item[$\overline{I}_e/ \underline{I}_e$]
		Current limits of EL
		
		\item[$K_w^{\mathrm{deload}}$]
		Deloading coefficient of WT
		
		\item[$n_{\mathrm{c}}$]
		Number of cells in stack
		
		\item[$\overline{P}_{e/g}/\underline{P}_{e/g}$]
		Power limits of EL/AFG
		
		\item[$\overline{P}_{ij}/\underline{P}_{ij}$]
		Power limits of branch $ij$
		
		\item[$P_{b}^{\mathrm{C/D,lim}}$]
		Charging/discharging power limit of BES
		
		\item[$q_{c,t}^\mathrm{lim}$]
		Flow limit of compressor
		
		\item[$r_g^{\mathrm{u/d}}$]
		Up/down ramping limits of AFG
		
		\item[$ R_e^\mathrm{ohm}$]
		Ohmic resistance of EL
		
		\item[$R_{g}^\mathrm{+/-,lim}$]
		Up/down regulation reserve limits of AFG
		
		\item[$R_{g}^\mathrm{PFR,lim}$]
		PFR reserve limit of AFG
		
		\item[$ t_{b/e/g/w} $]
		PFR delivery time of BES/EL/AFG/WT
		
		\item[$ t_{\mathrm{DB1/DB2}}$]
		Time to reach frequency deadband 1/2
		
		\item[$T_{g}^{\mathrm{on/off}}$]
		Limit of start-up/shutdown duration time of AFG
		
		\item[$ \overline{T}_e/ \underline{T}_e$]
		Operation temperature limits of EL
		
		\item[$ V^\mathrm{re/tn}$]
		Reverse/thermal-neutral voltage of EL cell
		%
		
		\item[$LHV_{\mathrm{NH_3}}$]
		Lower heating value of ammonia
		
		\item[$LHV_{\mathrm {H_2}}$]
		Lower heating value of hydrogen
		
		\item[$\rho_{w/r}$]
		Violation probability of chance constraint
		\item[$\eta_{e}^{\mathrm{pump/cool}}$]
		Efficiency of pump / cooling system
		
		\item[$\eta_g^\mathrm{comb/steam}$]
		Efficient of combustion/steam turbine

		\item[$\eta_b^{\mathrm{BES}}$]
		Efficiency of charging/discharging of BES
	\end{description}
	
\end{multicols}

\break
\section{Introduction}
\label{sec:intro}

\subsection{Background and motivation}
\label{sec:Background and motivation}

Driven by decarbonization targets, renewable power-to-hydrogen (ReP2H) systems that couple local wind and solar generation with hydrogen production have become an important pathway for low-carbon transition in the hard-to-abate chemical and maritime sectors~\cite{li2025redesigning}. The \emph{International Renewable Energy Agency} projects that global electrolyzer capacity will reach 500~GW by 2050, with 66\% of hydrogen produced from renewable-powered electrolysis~\cite{bianco_green_2020}.
In regions rich in renewable resources but lacking grid access (e.g., deserts, and plateaus), utility-scale off-grid ReP2H systems are emerging to produce hydrogen and downstream green fuels (e.g., ammonia, and methanol) \cite{zhu2026exploring,zeng2025planning,Yu2024planning}. Pilot 
projects have been commissioned in Songyuan and Da'an in China~\cite{Songyuan2023,zeng2024investment}, as well as in Australia~\cite{arena_yuri}, Europe~\cite{ammoniaenergy2024}, and many other regions.

The converter-dominated nature and lack of external grid support, however, makes off-grid ReP2H systems highly vulnerable to transient frequency instability~\cite{lu2025stability}, which may not only disrupt power transmission across transformers and converters, but also endanger co-located chemical processes via rotating compressors and pumps. While grid-forming (GFM) control offers a promising solution \cite{zhu2026exploring}, the industrial deployment of GFM renewable generations and battery energy storage (BES) remains limited~\cite{osti_2478292}. Moreover, relying solely on GFM-BES for frequency support
may accelerate degradation and reduce lifecycle performances~\cite{zhu2026exploring}. As a practical alternative, new system designs retrofit auxiliary generators of co-located chemical plants into ammonia-fueled generators (AFGs) ~\cite{Longyan2025,kong2024real}. They can provide rotational inertia and primary frequency regulation (PFR), while also serving as back-up power sources for critical chemical processes in the hydrogen downstream.

To maintain frequency security under disturbances such as renewable fluctuations or electrolyzer outages, frequency-constrained (FC) scheduling methods can be applied to dispatch AFGs in off-grid ReP2H systems~\cite{chu2021frequency,chu2024scheduling,cui2024control}. These methods convert post-contingency frequency metrics, such as frequency nadir, rate of change of frequency (RoCoF), and quasi-steady-state frequency, into constraints on inertia and reserve capacity. For example, Chu et al.~\cite{chu2021frequency} derived nadir and RoCoF of center-of-inertia (COI) frequency from the swing equation and embedded them into microgrid scheduling as security constraints. Later studies extended this framework by incorporating wind turbines (WTs) and BESs for virtual inertia (VI) and PFR provision~\cite{chu2024scheduling}, and by exploring GFM-enabled BES for enhanced performance~\cite{cui2024control}.
These studies confirm that WTs and BESs can mitigate frequency deviations and reduce the regulation burden of generators. However, the frequency support potential of hydrogen plants (HPs) remains unexplored.

Utility-scale HPs consist of multiple electrolyzers operating under varying load, which can inherently provide upward and downward reserves~\cite{qiu2023extend}. Enabling HPs to contribute to frequency regulation could further reduce reliance on conventional reserves and lower operating costs. Nevertheless, HPs’ participation in FR alters system dynamics, and their frequency support depends on operating states and loading levels, creating a trade-off between hydrogen production and frequency security.

Aiming to enforce transient frequency security constraints while maximizing hydrogen production, this paper develops a frequency dynamic model for off-grid ReP2H systems, derives stage-wise transient frequency security constraints, and proposes a frequency-supporting production scheduling (FSPS) approach to ensure reliable frequency support from HPs in off-grid ReP2H systems. Related literature are reviewed in Section \ref{sec:review}, and Section \ref{sec:contribution} outlines the main contributions of this work.

\subsection{Literature review}
\label{sec:review}

Experimental studies have shown that alkaline water electrolyzers (AWEs) and proton exchange membrane electrolyzers (PEMELs) can deliver second-scale (0.1--4~s) load ramping~\cite{entsoe2022p2h2,qiu2023dynamic,cheng2025power}. Their potential for grid frequency regulation has therefore attracted increasing attention. A summary of recent studies is provided in Table~\ref{tab:literature}.

For example, Hossain et al.~\cite{hossain2023power} designed a frequency controller and demonstrated fast PFR from PEMELs. Dozein et al.~\cite{dozein2021fast,dozein2022virtual} proposed control schemes that enable PEMELs to provide both PFR and VI response. Huang~et~al.~\cite{huang2022analytical} developed a controller allowing AWEs to participate in PFR. Guan~et~al.~\cite{guan2024frequency} further investigated parameter tuning for AWEs' frequency control.
Cheng et al.~\cite{cheng2023coodinated} coordinated electrolyzers and WTs to deliver frequency support to the utility grid and participate in the frequency regulation market. Wu et al.~\cite{wu2022incentivizing} and Li et al.~\cite{li2025pricing,li2024cvar} incorporated electrolyzer-based inertia and PFR into FC scheduling for bulk power systems.

\begin{table}[t]
\centering
\renewcommand{\arraystretch}{1.05}
\caption{Summary and comparison of existing studies on frequency support from hydrogen plants.}
\label{tab:literature}
\vspace{-6pt}
\centering
\begin{adjustbox}{width=\textwidth}
\begin{tabular}{ccccccccccc}
\hline\hline
\multirow{3}{*}{Literature\vspace{-4pt}}
& \multirow{3}{*}{\tabincell{c}{Frequency \\ support \\ from HP\vspace{-4pt}}}
& \multirow{3}{*}{\tabincell{c}{HP model}\vspace{-4pt}}
& \multicolumn{4}{c}{\tabincell{c}{Frequency support types}}
& \multicolumn{2}{c}{\tabincell{c}{Frequency model}}
& \multirow{3}{*}{\tabincell{c}{Coupled \\ HP \\ scheduling \vspace{-4pt}}}
& \multirow{3}{*}{\tabincell{c}{Uncertainty}\vspace{-4pt}}
\\
\cline{4-7}
\cline{8-9}
& &
& Inertia & Damping & PFR & GFM
& \tabincell{c}{Staged FR$^{2}$} & \tabincell{c}{Staged\\ security\\  constraints$^{3}$}
& &
\\ \hline

Chu et al.~\cite{chu2021frequency,chu2024scheduling}
& $\times$ & --
& $\checkmark$ & $\checkmark$ & $\checkmark$ & $\times$
& $\times$ & $\times$
& $\times$ & $\checkmark$
\\

Cui et al.~\cite{cui2024control}
& $\times$ & --
& $\checkmark$ & $\checkmark$ & $\checkmark$ & $\checkmark$
& $\times$ & $\times$
& $\times$ & $\checkmark$
\\

Hossain et al.~\cite{hossain2023power}
& $\checkmark$ & Single-PEMEL
& $\times$ & $\times$ & $\checkmark$ & $\times$
& $\times$ & $\times$
& $\times$ & $\times$
\\

Dozein et al.~\cite{dozein2021fast}
& $\checkmark$ & Single-AWE
& $\times$ & $\times$ & $\checkmark$ & $\times$
& $\times$ & $\times$
& $\times$ & $\times$
\\

Dozein et al.~\cite{dozein2022virtual}
& $\checkmark$ & Single-PEMEL
& $\checkmark$ & $\checkmark$ & $\times$ & $\times$
& $\times$ & $\times$
& $\times$ & $\times$
\\

Huang et al.~\cite{huang2022analytical}
& $\checkmark$ & Multi-AWE
& $\times$ & $\times$ & $\checkmark$ & $\times$
& $\times$ & $\times$
& $\times$ & $\times$
\\

Guan et al.~\cite{guan2024frequency}
& $\checkmark$ & Multi-AWE
& $\times$ & $\times$ & $\checkmark$ & $\times$
& $\times$ & $\times$
& $\times$ & $\times$
\\

Cheng et al.~\cite{cheng2023coodinated}
& $\checkmark$ & Aggregated EL$^{1}$
& $\times$ & $\times$ & $\checkmark$ & $\times$
& $\times$ & $\times$
& $\times$ & $\times$
\\

Wu et al.~\cite{wu2022incentivizing}
& $\checkmark$ & Aggregated EL
& $\checkmark$ & $\times$ & $\checkmark$ & $\times$
& $\checkmark$ & $\times$
& $\times$ & $\times$
\\

Li et al.~\cite{li2025pricing}
& $\checkmark$ & Aggregated EL
& $\checkmark$ & $\times$ & $\checkmark$ & $\times$
& $\checkmark$ & $\times$
& $\times$ & $\checkmark$
\\

Li et al.~\cite{li2024cvar}
& $\checkmark$ & Aggregated EL
& $\checkmark$ & $\checkmark$ & $\checkmark$ & $\times$
& $\checkmark$ & $\times$
& $\times$ & $\checkmark$
\\
\hline
\textbf{This work}
& $\checkmark$
& \makecell[l]{\tabincell{c}{\textbf{Multiple AWEs}\\ \textbf{and PEMELs}}}
& \checkmark & \checkmark & \checkmark & \checkmark
& \checkmark & \checkmark
& \checkmark & \checkmark
\\
\hline\hline

\multicolumn{11}{l}{\makecell[l]{$^{1}$Electrolyzer; $^{2}$Explicit modeling of staged frequency regulation, i.e., resources response under different dead-bands and time delays.}}\\
\multicolumn{11}{l}{\makecell[l]{$^{3}$Frequency security constraints capture the impact of the staged frequency regulation strategy (e.g., nadir occurrence in different stages). }}

\end{tabular}
\end{adjustbox}
\end{table}

Despite these advances, existing approaches are not directly applicable to utility-scale off-grid ReP2H systems, with the limitations given below:

First, off-grid ReP2H systems may integrate heterogeneous resources, including AWEs, PEMELs, BESs, WTs, and AFGs, each with distinct dynamic characteristics. Their coordinated response under disturbances leads to complex frequency behavior. Many existing models assume instantaneous PFR delivery from electrolyzers~\cite{li2024cvar} or neglect load damping effects~\cite{wu2022incentivizing}, which limits their accuracy and applicability for off-grid ReP2H systems that are very sensitive to disturbances. A refined yet tractable frequency dynamic model that captures staged regulation from diverse resources is still lacking.

Second, most FC scheduling studies model the HP as a single aggregated electrolyzer with fixed VI and PFR capability over the scheduling horizon \cite{cheng2023coodinated,wu2022incentivizing,li2025pricing,li2024cvar}. This simplification neglects both the heterogeneity of frequency regulation capabilities among different types of electrolyzers in utility-scale multi-electrolyzer HPs and the influence of production scheduling, i.e., on/off switching and load allocation, on the frequency support performance of the HP.

Specifically, HPs may contain both AWEs and PEMELs~\cite{hu2025triple}. Due to slower electrochemical kinetics and stronger electrical double-layer (EDL) effects, AWEs exhibit slower power response than PEMELs~\cite{agredano2025decentralized,dang2022transient}. PEMELs, in contrast, provide faster and more flexible regulation, but are smaller in capacity \cite{entsoe2022p2h2,cheng2025power}. Representing the entire plant as a single aggregated unit, therefore, misestimates its actual capability.

Moreover, electrolyzers in utility-scale HPs are scheduled using rotation rules~\cite{li2023exploration,zheng2023off}  or optimization-based strategies~\cite{li2024two,zheng2022optimal} to maximize renewable utilization and hydrogen output. The available VI and PFR reserves vary with the on/off states and loading levels of individual electrolyzers. Because production efficiency depends on operating conditions such as load level and temperature~\cite{zheng2022optimal}, ensuring frequency security requires balancing hydrogen output and reserve provision. To date, limited studies have jointly modeled differentiated frequency support from AWEs and PEMELs within plant-level production scheduling, to the best of our knowledge.

\subsection{Contributions}
\label{sec:contribution}

To address these gaps, this paper proposes an FC production scheduling framework for utility-scale off-grid ReP2H systems. The proposed method jointly optimizes off-grid network operation and multi-electrolyzer HP production scheduling, enabling part-loaded AWEs and PEMELs to provide reliable frequency support based on their assigned reserves and fast response. This support partially replaces conventional reserves from AFGs and renewable units while maintaining profitability from hydrogen output and transient frequency security. The main contributions include:
\begin{enumerate}
	\item We develop a detailed frequency response model of HPs that considers distinct frequency support capabilities of AWEs and PEMELs, and establish a multi-stage frequency dynamic model for the off-grid ReP2H system. Based on this model, we derive stage-wise transient frequency security constraints and show that they can be incorporated into scheduling with acceptable accuracy.

	\item We propose a frequency-supporting production scheduling (FSPS) approach that jointly considers VI and PFR provision from AWEs and PEMELs, as well as their nonlinear efficiency and thermal characteristics. The proposed FSPS coordinates hydrogen production and reserve allocation to ensure adequate frequency support while maximizing benefits from hydrogen production.
	
	\item We embed the FSPS into the FC scheduling framework of the off-grid ReP2H system and formulate a joint optimization model. The model captures the interaction between plant-level reserve provision and system-level operation cost, maximizes net profit, and enforces transient frequency security. Year-long simulation shows that the proposed method improves annual net profit by 28.96\% on average, attributed to the reduced FR costs while maintaining comparable hydrogen output.
\end{enumerate}

The remainder of this paper is organized as follows. Section~\ref{sec:FreqResponse} presents the frequency response models. Section~\ref{sec:FreqDynamic} derives the system frequency dynamics and security constraints. Section~\ref{sec:SchedulingModel} introduces the FSPS-embedded FC scheduling model. Section~\ref{sec:CaseStudy} reports case studies, and Section~\ref{sec:Conclusion} concludes the paper.

\begin{figure}[t]
	\centering
	\includegraphics[width=5.5in]{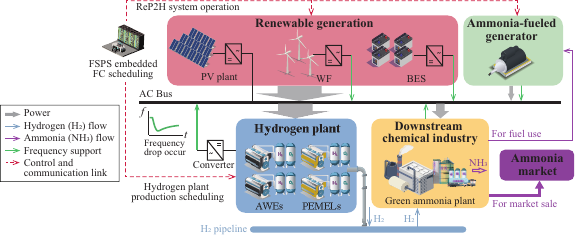}\vspace{-4.5pt}
	\caption{Schematic diagram of a typical off-grid ReP2H system.}
	\label{fig:system}\vspace{0pt}
\end{figure}

\section{Modeling of frequency response from heterogeneous regulation resources in ReP2H systems}
\label{sec:FreqResponse}

A typical off-grid ReP2H system is illustrated in Fig.~\ref{fig:system}. In this system, HPs are the main electrical loads and convert most renewable electricity into hydrogen for ammonia synthesis. The majority of the produced ammonia is exported to the market to supply the chemical and shipping industries, while the remainder serves as fuel for AFGs. Since both generation and electrolytic loads are connected through power-electronics converters, the system has low physical inertia and is therefore sensitive to disturbances in terms of frequency security.

Although resources, e.g., renewable inverters, battery energy storage, electrolyzers and AFGs, can provide frequency regulation, their response differ substantially. AFGs rely on rotor inertia and governor dynamics, while renewable units, batteries, and electrolyzers respond through converter control. This section characterizes the dynamic frequency support capability of the heterogeneous resources and formulates a coordinated regulation strategy.

For clarity, let $\mathcal{A}, \mathcal{P}, \mathcal{E},\mathcal{B}, \mathcal{W}, \mathcal{G}$ denote the sets of AWEs, PEMELs, all electrolyzers, BESs, WTs, and AFGs, respectively, with $\mathcal{E} = \mathcal{A} \cup \mathcal{P}$.

\subsection{Frequency support from HPs}

AWEs and PEMELs are the dominant industrial electrolyzers. Given an electrolysis current $I$, hydrogen production rate and stack power of both AWEs and PEMELs can be expressed as \cite{zeng2024scheduling,zheng2022optimal}:
\begin{gather}
	q_{e}^{\mathrm{H_2}}  = \frac{\eta_{\mathrm{F}}(I,T) n_{\mathrm{c}} I M^{\mathrm{H_2}}} { 2F},  \\
	P_{e}^{\mathrm{stack}}(I,T)  = I V^{\mathrm{stack}}(I,T), \label{eq:7}
\end{gather}
where $M^{\mathrm{H_2}}$ is hydrogen molar mass; $\eta_{\mathrm{F}}(\cdot) =  (I/A)^2/(f_1+(I/A)^2)\times f_2$ is Faraday efficiency, with $f_1 = 2.5T + 50$, $f_2 = 1 - 6.25 \times 10^{-6}T$; and  $F = 96485.3 \ \mathrm{C/mol}$ is Faraday constant. For detailed expressions of $\eta_{\mathrm{F}}$ and $V^{\mathrm{stack}}$, readers are referred to our prior work~\cite{qiu2023dynamic,zeng2024scheduling}.

Electrolyzers provide frequency support by adjusting load in response to the locally measured frequency signal, as shown in Fig.~\ref{fig:StepResponse}(a). Since thermal dynamics evolve on an hour scale, they are neglected in second-level frequency analysis. The equivalent circuit for AWEs and PEMELs in Fig.~\ref{fig:StepResponse}(b), which captures dominant short-term electrical dynamics and has been validated against experiments \cite{sha2025semi,sha2023low,cheng2025power}, captures the EDL effect through $C_e^{\mathrm{dl}}$, $R_e^{1}$, and $R_e^{2}$. Based on Fig. \ref{fig:StepResponse}(b), the transient stack power under current $I(\tau)$ in terms of transient time $\tau$ can be expressed as:
\begin{gather}
	P_e^{\mathrm{stack}}\big(I(\tau)\big)  = n_{\mathrm{c}}V_{e}^{\mathrm{stack}}\big(I(\tau)\big)I(\tau),  \label{eq:10} \\
	V_e^{\mathrm{stack}}(I(\tau))  =V_e^{\mathrm{re}}+V_e^{\mathrm{dl}}(\tau)+R_e^{\mathrm{ohm}}I(\tau), \\
	C_e^{\mathrm{dl}} \frac{\mathrm{d}V_e^{\mathrm{dl}}(\tau)}{\mathrm{d}\tau}=I(\tau)-V_e^{\mathrm{dl}}(\tau)/(R_e^\mathrm{1}+R_e^\mathrm{2}).  \label{eq:12}
\end{gather}

Industrial-scale electrolyzers commonly adopt current-following control in their rectifiers~\cite{cheng2025power}. Under a step change in the current reference signal at \(t = 0\), with \(I(0) = I\) and \(I(\infty) = I + \Delta I\), assuming ideal rectifier current tracking, the power response of the electrolysis stack can be derived from (\ref{eq:10})--(\ref{eq:12}) as:
\begin{gather}
	\Delta P_e^{\mathrm{stack}}(\tau)=  n_{\mathrm{c}} \Delta I[V_e^{\mathrm{re}}+\Delta I(R_e^\mathrm{1}+R_e^\mathrm{2})+\Delta IR_e^{\mathrm{ohm}}]  +n_{\mathrm{c}} \Delta I[V_e^{\mathrm{dl}}(0)-\Delta I(R_e^\mathrm{1}+R_e^\mathrm{2})]e^{(-\tau/\theta)},
\end{gather}
where $ \theta=(R_e^1+R_e^2)C_e^{\mathrm{dl}}$ is the time constant associated with the EDL effect. AWEs typically have larger $C_e^{\mathrm{dl}}$ due to the liquid electrolytes and lower current density, and thus slower response than PEMELs~\cite{cheng2025power}.

\begin{figure}[t]
	\centering
	\includegraphics[width=5.5in]{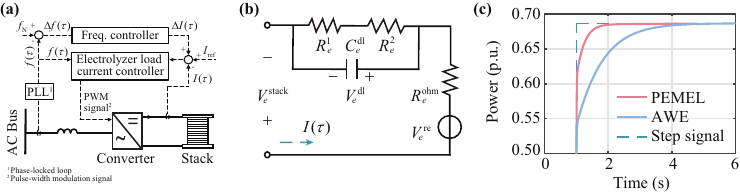}\vspace{-4.5pt}
	\caption{Frequency control and dynamic load response of electrolyzers. (a) Frequency control diagram. (b) Equivalent circuit of the electrolysis stack. (c) Step responses of the AWE and PEMEL.}
	\label{fig:StepResponse}\vspace{0pt}
\end{figure}

Fig.~\ref{fig:StepResponse}(c) shows the simulated load step responses of 5 MW-rated AWE and PEMEL with parameters in~\cite{cheng2025power,zhu2026exploring}. The rise times for AWE load power to reach 50\% and 95\% of the reference are about 1~s and 2.8~s, respectively, while those for PEMELs are only 0.04~s and 0.55~s. During frequency regulation, inertial response power must be delivered before the frequency deviation reaches the deadband time \(t_{\mathrm{DB}}\) to limit RoCoF. In large power grids, \(t_{\mathrm{DB}} \approx 0.5~\mathrm{s}\)~\cite{trovato2018unit}, whereas in low-inertia microgrids such as the off-grid ReP2H system, \(t_{\mathrm{DB}} < 0.5~\mathrm{s}\). Thus, only PEMELs can provide effective VI support, and the equivalent inertia can be modeled by matching the maximum RoCoF constraint, as:
\begin{align}
	H_e=\frac{P_e^{\mathrm{stack}}(t_\mathrm{db})-P_e^{\mathrm{stack}}(0)}{\Delta f_{\mathrm{RoCoF}}^{\mathrm{lim}}}/f_0=\frac{\Delta P_e^{\mathrm{VI}}}{\Delta f_{\mathrm{RoCoF}}^{\mathrm{lim}}f_0},
\end{align}
where $\Delta P_e^{\mathrm{VI}}$ denotes the inertial response power of PEMELs under the maximum current ramping. In contrast, AWEs primarily contribute to PFR due to slower dynamics. The overall frequency response of the electrolyzers in an HP follows:
\begin{align}
	\Delta P_{\mathrm{HP}}^{\mathrm{stack}}(\tau) =
	\sum_{e \in \mathcal{P}} \Delta P_{e}^{\mathrm{VI}}(\tau) +
	\sum_{e \in \mathcal{A}} \Delta P_{e}^{\mathrm{PFR}}(\tau).
\end{align}

The load power of auxiliary devices, like hydrogen compressors, in the HP also exhibits frequency dependence, thereby the overall frequency response of the HP can therefore be expressed as:
\begin{align}
	\Delta P_{\mathrm{HP}}\left(\tau\right)=\Delta P_{\mathrm{HP}}^{\mathrm{stack}}\left(\tau\right)+ \sum_{e \in \mathcal{E}} \Delta P_{e}^{\mathrm{auxi}}(\tau) + \sum_{c \in \mathcal{C}} \Delta P_{c}^{\mathrm{comp}}(\tau),  \label{eq:16}
\end{align}
where $ \mathcal{C} $ represent the set of hydrogen compressors, and
\begin{align}
	\{\Delta P_{e}^{\mathrm{auxi}}(\tau),\Delta P_{c}^{\mathrm{comp}}(\tau)\} = D_{d} \Delta f_{e}(\tau) \{P_{e}^{\mathrm{auxi}}, P_{c}^{\mathrm{comp}}\}.
\end{align}

Note that the available frequency reserve of each electrolyzer depends on its scheduled operating point. Since the stack efficiency $\eta^{\mathrm{stack}}=(q_e^{\mathrm{H_2}}LHV_\mathrm{H_2})/P_e^{\mathrm{stack}}(I,T)$ varies with its current load $I$ and temperature $T$, as shown in Fig.~\ref{fig:efficiency}. Electrolyzers are therefore scheduled near high-efficiency regions \cite{qiu2023extend}. This scheduling results, including on/off state and load level of each electrolyzer, directly determine both hydrogen output and available VI and PFR reserves.
Production scheduling thus directly affects frequency support capability, creating a need to balance hydrogen output and reserve provision.

\begin{figure}[t]
	\centering
	\includegraphics[width=3.75in]{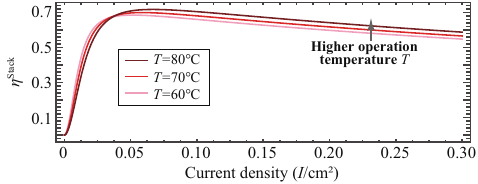}\vspace{-4.5pt}
	\caption{AWE Stack efficiency with different current density and temperature.}
	\label{fig:efficiency}\vspace{0pt}
\end{figure}

\subsection{Frequency support from AFGs, BES and downstream chemical plant}

\subsubsection{Ammonia-fueled generator (AFG)}

AFGs provide inertia, damping, and PFR through rotor and governor dynamics under frequency disturbances \cite{chu2021frequency}. Their frequency response is modeled as:
\begin{align}
	\hspace{-2pt}
	& \Delta P_{g}(\tau)=-2H_{g}\frac{\mathrm{d}\Delta f_{g}(\tau)}{\mathrm{d}\tau}-D_{g}\Delta f_{g}(\tau)+\Delta P_{g}^{_{\mathrm{PFR}}}(\tau). \label{eq:1}
\end{align}

\subsubsection{Wind turbine (WT)}

WTs can operate in a de-loading mode to create reserve margins and release reserved power to provide PFR in response to frequency deviations. The response model is written as:
\begin{gather}
	\tilde{P}_{w}^{\mathrm{avail}}=(1{-}k_{w}^{\mathrm{deload}})\tilde{P}_{w}=(1{-}k_{w}^{\mathrm{deload}})(\hat{P}_{w}+\xi_{w}), \label{eq:2}  \\
	\Delta P_{w}^{\mathrm{PFR}}(\tau) \leq \Delta P_{w}^{\mathrm{deload}}=k_{w}^{\mathrm{deload}}\tilde{P}_{w}=k_{w}^{\mathrm{deload}}\hat{P}_{w}+\xi_{w}^{'}. \label{eq:3}
\end{gather}
However, WT's PFR capability is subject to the uncertain forecast error $\xi_{w}$, which will be addressed using an uncertainty-aware tool later introduced in Section \ref{sec:DRCC-CO}.


\subsubsection{Battery energy storage (BES)}

BESs can provide frequency support under both frequency droop control and GFM control in off-grid ReP2H systems \cite{zhu2026exploring,lu2025stability}. The former enables BES to provide PFR, and the latter allows it to emulate the rotor swing dynamics of generators, thereby actively providing VI and virtual damping support. Thus, BES's frequency response is modeled as:
\begin{align}
	\Delta P_{b}(\tau)&=\Delta P_{b}^{\mathrm{PFR}}(\tau) + \Delta P_{b}^{\mathrm{GFM}}(\tau),  =\Delta P_{b}^{\mathrm{PFR}}(\tau)-\Big[2H_{b}\frac{\mathrm d\Delta f_{b}(\tau)}{\mathrm d\tau}+D_{b}\Delta f_{b}(\tau)\Big].  \label{eq:4}
\end{align}

\subsubsection{Downstream chemical plant}

The frequency response of the downstream ammonia plant originates from the frequency-dependent behavior of large-capacity electric drives including syngas compressors and pumps~\cite{zeng2025planning,yu2024ammoniaoptimal}. This response can be approximated as:
\begin{align}
	\hspace{-2pt}
	\Delta P_{d}(\tau) = D_d \Delta f(\tau) P_d. \label{eq:5}
\end{align}

\subsection{Multi-stage frequency regulation strategy}


Inertial response is instantaneous, whereas PFR is implemented by assigning differentiated deadbands for staged participation~\cite{trovato2018unit}.
To prioritize fast resources, we assign smaller deadbands to HPs and BES than to WTs and AFGs. This design enhances early frequency containment and reduces reserve requirements of slower units, thereby aligning with the economic requirements of ReP2H systems. Under this strategy, the system's total PFR is composed of:
\begin{align}
	\Delta P_{b}^{\mathrm{PFR}} &=
	\begin{cases}
		0, & \tau \le t_{\mathrm{DB1}},\\
		R_{b}^{\mathrm{PFR}}({\tau - t_{\mathrm{DB1}}})/{t_{b}}, &   t_{\mathrm{DB1}}<\tau\le t_{b}+t_{\mathrm{DB1}}, \\
		R_{b}^{\mathrm{PFR}}, & \tau > t_{b}+t_{\mathrm{DB1}},
	\end{cases}
	\label{eq:17}  \\
	\Delta P_{e}^{\mathrm{PFR}} &=
	\begin{cases}
		0, & \tau \le t_{\mathrm{DB1}},\\
		R_{e}^{\mathrm{PFR}}({\tau - t_{\mathrm{DB1}}})/{t_{e}}, & t_{\mathrm{DB1}}<\tau\le t_{e}+t_{\mathrm{DB1}}, \\
		R_{e}^{\mathrm{PFR}}, & \tau > t_{e}+t_{\mathrm{DB1}},
	\end{cases} \\
	\Delta P_{w}^{\mathrm{PFR}} &=
	\begin{cases}
		0, & \tau \le t_{\mathrm{DB2}},\\
		R_{w}^{\mathrm{PFR}}({\tau - t_{\mathrm{DB2}}})/{t_{w}}, & t_{\mathrm{DB2}}<\tau\le t_{w}+t_{\mathrm{DB2}},\\
		R_{w}^{\mathrm{PFR}}, & \tau > t_{w}+t_{\mathrm{DB2}},
	\end{cases} \\
	\Delta P_{g}^{\mathrm{PFR}} &=
	\begin{cases}
		0, & \tau \le t_{\mathrm{DB2}},\\
		R_{g}^{\mathrm{PFR}}({\tau - t_{\mathrm{DB2}}})/{t_{g}}, & t_{\mathrm{DB2}}<\tau\le t_{g}+t_{\mathrm{DB2}},\\
		R_{g}^{\mathrm{PFR}}, & \tau > t_{g}+t_{\mathrm{DB2}},
	\end{cases}
	\label{eq:20}
\end{align}
where \(t_e > t_b > (t_{\mathrm{DB2}} - t_{\mathrm{DB1}})\) and \(t_g > t_w\).
The PFR models in (\ref{eq:17})--(\ref{eq:20}) follow the assumption that the early post-contingency frequency deviation is approximately linear with time, which has been validated by previous studies~\cite{chu2021frequency,chu2024scheduling,cui2024control,wu2022incentivizing}.
The resource-specific PFR delivery times \(\{t_b, t_e, t_w, t_g\}\) in these models well preserve the respective frequency response characteristics of heterogenous resources, and exclusively designed multi-stage frequency regulation is captured by the deadband thresholds \(t_{\mathrm{DB1}}\) and \(t_{\mathrm{DB2}}\).

\section{Modeling of frequency dynamics and security metrics}
\label{sec:FreqDynamic}

\subsection{Staged frequency dynamics under disturbance}

Based on the frequency response models and the staged regulation strategy developed in Section~\ref{sec:FreqResponse}, this section establishes a unified frequency dynamic model for the off-grid ReP2H system and derives analytical expressions for transient frequency security metrics. We further demonstrate that, with engineering-acceptable accuracy, these metrics can be reformulated as tractable constraints and embedded into the ReP2H scheduling problem.

We begin by characterizing the system frequency response to a contingency-induced power disturbance. When a disturbance $\Delta P_{\mathrm{dis}}$ occurs, the frequency deviation in the off-grid ReP2H system triggers coordinated responses from AFGs, WTs, BESs, HPs, and the downstream chemical plant. The power balance during the transient process satisfies:
\begin{align}
	&\sum_{g \in \mathcal{G}} \Delta P_{g}(\tau) +  \sum_{w\in \mathcal{W}} \Delta P_{w}(\tau) + \sum_{b\in \mathcal{B}} \Delta P_{b}(\tau) + \Delta P_{\mathrm{HP}}\left(\tau\right) + \Delta P_{d}(\tau) = \Delta P_{\mathrm{dis}}, \label{eq:21}
\end{align}
\noindent
where a positive $\Delta P_{\mathrm{dis}}$ means loss of generation or increase of load. By substituting (\ref{eq:1})--(\ref{eq:5}) and (\ref{eq:17})--(\ref{eq:20}) into (\ref{eq:21}), the frequency dynamics during different PFR stages are obtained by:
\begin{align}
	&		2H \frac{\mathrm d\Delta f(\tau)}{\mathrm d\tau} + D\Delta f(\tau)= \label{eq:22}\\
	&\begin{cases}
		-\Delta P_{\text{dis}}, \
		\tau < t_{\mathrm{DB1}}, \\[0.5em]
		-\Delta P_{\text{dis}}
		+ \displaystyle\sum_{e\in \mathcal{A}} R_{e}^{\mathrm{PFR}}
		\frac{\tau - t_{\mathrm{DB1}}}{t_{e}}
		+ \displaystyle\sum_{b\in \mathcal{B}} R_{b}^{\mathrm{PFR}}
		\frac{\tau - t_{\mathrm{DB1}}}{t_{b}}, \
		t_{\mathrm{DB1}} \le \tau < t_{\mathrm{DB2}}, \\[0.5em]
		-\Delta P_{\text{dis}}
		+ \displaystyle\sum_{e\in \mathcal{A}} R_{e}^{\mathrm{PFR}}
		\frac{\tau - t_{\mathrm{DB2}}}{t_{e}}
		+ \displaystyle\sum_{b\in \mathcal{B}} R_{b}^{\mathrm{PFR}}
		\frac{\tau - t_{\mathrm{DB2}}}{t_{b}}
		+ \displaystyle\sum_{g\in \mathcal{G}} R_{g}^{\mathrm{PFR}}
		\frac{\tau - t_{\mathrm{DB2}}}{t_{g}}
		+ \displaystyle\sum_{w\in \mathcal{W}} R_{w}^{\mathrm{PFR}}
		\frac{\tau - t_{\mathrm{DB2}}}{t_{w}}, \
		\tau \ge t_{\mathrm{DB2}}. \notag
	\end{cases}
\end{align}
	
\noindent where system inertia is defined as
\begin{align}
 H = \sum_{b\in \mathcal{B}} H_{b} + \sum_{e\in \mathcal{P}} x_{e}^{\mathrm{st}}H_{e} + \frac{\sum_{g\in \mathcal{G}} x_{g}^{\mathrm{st}}H_{g} \,\overline{P}_{g}}{f_{\mathrm{N}}},
\end{align}
\noindent
where $f_{\mathrm{N}}$ is the nominal frequency; $\Delta f(\tau )$ is the COI frequency, which is expressed as
\begin{align}
	\Delta f(\tau )=\dfrac{1}{H} \Big[ \sum_{b\in \mathcal{B}}{{{H}_{b}}}\Delta {{f}_{b}}(\tau )+\sum_{e\in \mathcal{P}}x_{e}^{\mathrm{st}}{{{H}_{e}}}\Delta {{f}_{e}}(\tau )+\dfrac{\sum\nolimits_{g\in \mathcal{G}}x_{g}^{\mathrm{st}}{{{H}_{g}}\overline{P}_{g}}}{{{f}_{\mathrm{N}}}}\Delta {{f}_{g}}(\tau) \Big];
\end{align}
Since off-grid ReP2H systems generally have a short distance between generation and load \cite{zhu2026exploring,zeng2024scheduling,zeng2024investment}, frequency deviations across buses remain tightly clustered around the COI frequency~\cite{badesa2021conditions,ducoin2024analytical}. Hence, the COI frequency is adopted to represent the system-wide frequency dynamics for deriving the frequency security constraints;
$D$ is the overall system-level damping coefficient, expressed as $D={{D}_{d}}{P}_{d}+{{D}_{b}}$.

Fig.~\ref{fig:FreqDynamicSim} compares the post-contingency frequency dynamics and PFR power responses obtained from~\eqref{eq:22} with those from the electromagnetic transient (EMT) simulation benchmark developed form \cite{zhu2026exploring,cheng2025power}. It can be observed that the proposed model closely matches the benchmark in terms of maximum RoCoF, quasi-steady-state frequency deviation, and frequency nadir. The PFR power response of AWEs and AFGs depicted in Figs.~\ref{fig:FreqDynamicSim}(b) and (c) show that the PFR approximation models~\eqref{eq:17}--\eqref{eq:20} effectively capture the saturation effects and deadband-induced time delay in the frequency control loop. These results demonstrate that the proposed model~\eqref{eq:22} provides a satisfactory estimation of post-contingency frequency dynamics. Meanwhile, the slightly conservative approximation of frequency nadir and PFR power responses enhances robustness for practical operation.
\begin{figure}[t]
	\centering
	\includegraphics[width=5.5in]{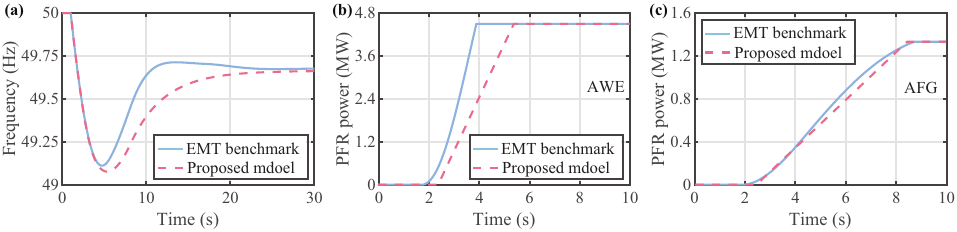}\vspace{-4.5pt}
	\caption{Comparison of the post-contingency frequency dynamics and PFR power under the proposed model~\eqref{eq:22} and the EMT benchmark. (a) Post-contingency frequency. (b) PFR power of AWEs. (c) PFR power of AFGs.}
	\label{fig:FreqDynamicSim}\vspace{0pt}
\end{figure}

\subsection{Transient frequency security metrics and constraints}

Based on (\ref{eq:22}), the transient frequency security metrics (maximum RoCoF, quasi-steady-state frequency, and frequency nadir) can be derived.

\begin{remark}
	The frequency nadir may occur uniquely in either the interval \([t_{\mathrm{DB1}}, t_{\mathrm{DB2}})\) or \([t_{\mathrm{DB2}}, \infty)\), as shown in Fig.~\ref{fig:FrequencyInterval}(a). The frequency dynamics in these intervals follow the piecewise formulations in \eqref{eq:22}, requiring two distinct expressions to characterize the nadir in each case.
\end{remark}

Here, we derive the expressions of the frequency nadir for both intervals and then propose a strategy to enforce only the unique nadir constraint in each optimization step.

\emph{1) RoCoF:} The maximum RoCoF occurs immediately after the disturbance, where the frequency deviation $\Delta f(\tau) \approx 0$ within the measurement window. Based on the first segment of \eqref{eq:22}, the RoCoF constraint is expressed as:
\begin{align}
  H \geq \dfrac{\Delta {{P}_{\mathrm{dis}}}} {2\Delta f_{\mathrm{RoCoF}}^{\mathrm{lim}}}. \label{eq:23}
\end{align}

\emph{2) Quasi-steady-state frequency:} As $\tau \to \infty$, the frequency reaches a quasi-steady state governed by the third segment of~\eqref{eq:22}. Applying the boundary condition $\mathrm{d}\Delta f/\mathrm{d}\tau=0$ gives
\begin{align}
  &\sum_{e\in \mathcal{A}}
  R_{e}^{\mathrm{PFR}}+
    \sum_{b\in \mathcal{B}}
  R_{b}^{\mathrm{PFR}}+
   \sum_{w\in \mathcal{W}}
  R_{w}^{\mathrm{PFR}} +
    \sum_{g\in \mathcal{G}}R_{g}^{\mathrm{PFR}} \ \ge \Delta {{P}_{\mathrm{dis}}} - D\Delta f_{\mathrm{qss}}^{\mathrm{lim}}. \label{eq:24}
\end{align}

\begin{figure}[t]
	\centering
	\includegraphics[width=4 in]{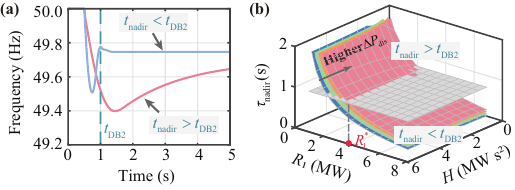}\vspace{-4.5pt}
	\caption{Frequency nadir dynamics. (a) Occurrence cases of the frequency nadir. (b) Time of the frequency nadir with varying reserve and inertia levels. }
	\label{fig:FrequencyInterval}\vspace{0pt}
\end{figure}

\emph{3) Frequency nadir:} The nadir may occur within one of two time intervals, i.e., \([t_{\mathrm{DB1}}, t_{\mathrm{DB2}})\), and \([t_{\mathrm{DB2}}, \infty)\). This leads to two set of nadir constraints.

The frequency nadir constraint within \([t_{\mathrm{DB1}}, t_{\mathrm{DB2}})\) is first derived. Assuming the initial condition \(\Delta f(0) = 0\),  the explicit solution of \eqref{eq:22} yields the time-dependent frequency trajectory:
\begin{align}
  \Delta f(\tau)=& \pm f_{\mathrm{db1}}
  + \frac{\Delta P_{\mathrm{dis}}\!\big(1 - e^{-\frac{D(\tau-t_{\mathrm{DB1}})}{2H}}\big)
  + R_{1}(\tau-t_{\mathrm{DB1}})}{D}  - \dfrac{2H R_{1}\!\big(1 - e^{-\frac{D(\tau-t_{\mathrm{DB1}})}{2H}}\big)}{D^{2}},
  \quad t_{\mathrm{DB1}} \le \tau \le t_{\mathrm{DB2}},     \label{eq:25}
\end{align}
\noindent
where $ R_{\mathrm{1}} =\sum\nolimits_{e\in \mathcal{A}}({R_{e}^{\mathrm{PFR}}}/{{t_{e}}}) +\sum\nolimits_{b\in \mathcal{B}} ({R_{b}^{\mathrm{PFR}}}/{{t_{b}}})$.

The nadir occurs when \({\mathrm{d}\Delta f(\tau)}/{\mathrm{d}\tau}=0\), thus we have:
\begin{align}
  \tau_\mathrm{naidr} = t_{\mathrm{DB1}} - \frac{2H}{D}
  \ln\!\Big(\frac{2H R_{1}}{2H R_{1} + D\Delta P_{\mathrm{dis}}}\Big).  \label{eq:26}
\end{align}

Substituting \eqref{eq:26} into \eqref{eq:25}, the frequency nadir is obtained as:
\begin{align}
  \Delta f_{\mathrm{nadir}}
  = \Delta f_{\mathrm{db1}}
  + \frac{\Delta P_{\mathrm{dis}}}{D}
  + \frac{2H R_{1}}{D^{2}}
  \ln\!\left(\frac{2H R_{1}}{D\Delta P_{\mathrm{dis}} + 2H R_{1}}\right). \label{eq:27}
\end{align}

Let $\Delta f_{\mathrm{nadir}} \le \Delta f_{\mathrm{nadir}}^{\mathrm{lim}}$ as the frequency nadir constraint, then we have:
\begin{align}
2HR_1 \ln \left(\frac{2HR_1}{D\Delta P_{\mathrm{dis}} + 2HR_1}\right)
\le \ & D^{2}(\Delta f_{\mathrm{nadir}}^{\mathrm{lim}} - \Delta f_{\mathrm{db1}}) - D\Delta P_{\mathrm{dis}}.\label{eq:28}
\end{align}

A sufficient condition for satisfying inequality~\eqref{eq:28} can be expressed as~\cite{chu2021frequency,cui2024control}:
\begin{align}
&x_{1}^{*} - \left(\sum\nolimits_{b\in \mathcal{B}}H_{b}R_{1} + \sum\nolimits_{e\in \mathcal{P}}x_{e}^{\mathrm{st}}H_{e}R_{1} + \sum\nolimits_{g\in \mathcal{G}}x_{g}^{\mathrm{st}}H_{g}R_{1}\right) \leq 0, \label{eq:29}
\end{align}
where \(x_{1}^{*}\) is obtained by solving
\begin{equation}
2x_{1}^{*}\ln\left(\frac{2x_{1}^{*}}{D\Delta P_{\mathrm{dis}} + 2x_{1}^{*}}\right)
= D^{2}(\Delta f_{\mathrm{nadir}}^{\mathrm{lim}} - \Delta f_{\mathrm{db1}}) - D\Delta P_{\mathrm{dis}}.
\label{eq:32}
\end{equation}

Using a similar method, the frequency nadir constraint corresponding to the occurrence interval \([t_{\mathrm{DB2}}, \infty)\) can be derived as:
\begin{align}
& x_{2}^{*} - \Big(\sum_{b\in \mathcal{B}}H_{b}R_{2} + \sum_{e\in \mathcal{P}}x_{e}^{\mathrm{st}}H_{e}R_{2} + \sum_{g\in \mathcal{G}}x_{g}^{\mathrm{st}}H_{g}R_{2}\Big)   \leq 0, \label{eq:31}
\end{align}
where $ R_{2}
= \sum\nolimits_{e \in \mathcal{A}} \bigl( R_{e}^{\mathrm{PFR}} / t_{e} \bigr)
+ \sum\nolimits_{b \in \mathcal{B}} \bigl( R_{b}^{\mathrm{PFR}} / t_{b} \bigr)
+ \sum\nolimits_{g \in \mathcal{G}} \bigl( R_{g}^{\mathrm{PFR}} / t_{g} \bigr)
+ \sum\nolimits_{w \in \mathcal{W}} \bigl( R_{w}^{\mathrm{PFR}} / t_{w} \bigr)$; and the value of \(x_{2}^{*}\) is obtained by solving
\begin{equation}
2x_{2}^{*}\ln\left(\frac{2x_{2}^{*}}{D\Delta P_{\mathrm{dis}} + 2x_{2}^{*}}\right)
= D^{2}(\Delta f_{\mathrm{nadir}}^{\mathrm{lim}} - \Delta f_{\mathrm{db2}}) - D\Delta P_{\mathrm{dis}}.
\label{eq:34}
\end{equation}

Since the nadir occurs uniquely in one interval of \(\tau \in [t_{\mathrm{DB1}}, t_{\mathrm{DB2}})\) and \(\tau \in [t_{\mathrm{DB2}}, \infty)\), only one of~\eqref{eq:29} or~\eqref{eq:31} should be enforced. Proposition \ref{prop:1} establishes monotonicity properties of $\tau_{\mathrm{nadir}}(H,R_1)$ with respect to $R_1$, and then finds threshold values $R_1^{\mathrm{lo}}$ and $R_1^{\mathrm{hi}}$ to determine the active region.
\begin{proposition}
	\label{prop:1}
	Let $\tau_{\mathrm{nadir}}(H,R_1)$ denote~\eqref{eq:26} under a known $\Delta P_{\mathrm{dis}}$.
	For any $\Delta P_{\mathrm{dis}}>0$ and $H>0$, there exists a unique
	$R_1^*(H)>0$ such that $\tau_{\mathrm{nadir}}\big(H,R_1^*(H)\big)=t_{\mathrm{DB2}}$, as shown in Fig.~\ref{fig:FrequencyInterval}(b).
	Define
	$
	R_1^{\mathrm{lo}}:=R_1^*(\underline H),\
	R_1^{\mathrm{hi}}:=R_1^*(\overline H),
	$
	where $\underline H$ and $\overline H$ are the inertia bounds.
	Then, for all $H\in[\underline H,\overline H]$,
	\begin{gather}
		 R_1\ge R_1^{\mathrm{hi}}
		\ \Rightarrow\
		\tau_{\mathrm{nadir}}(H,R_1)\le t_{\mathrm{DB2}}, \\
		 R_1\le R_1^{\mathrm{lo}}
		\ \Rightarrow\
		\tau_{\mathrm{nadir}}(H,R_1)\ge t_{\mathrm{DB2}}.
	\end{gather}
	Moreover, there exists a small time margin $\mu\ge 0$ such that for any
	$R_1\in(R_1^{\mathrm{lo}},R_1^{\mathrm{hi}})$ and any $H\in[\underline H,\overline H]$,
	$
	\tau_{\mathrm{nadir}}(H,R_1)\ \ge\ t_{\mathrm{DB2}}-\mu.
	$
	Consequently, an approximate binary reformulation of the frequency nadir constraint can be written as:
	\begin{gather}
		 R_{1} \ge R_{1}^{\mathrm{hi}} - M(1-\delta),\\
		 R_{1} \le R_{1}^{\mathrm{hi}} + M\delta,\\
         \varphi_{1}(x) \le M(1-\delta), \label{eq:36}\\
		 \varphi_{2}(x) \le M\delta, \label{eq:37}
	\end{gather}
	where $\delta\in\{0,1\}$; $\varphi_{1}(x)$ and $\varphi_{2}(x)$ are the left-hand sides of~\eqref{eq:29} and~\eqref{eq:31}, respectively.
\end{proposition}

\begin{proof}
	For any $R_1>0$ and $H>0$, differentiating $\tau_{\mathrm{nadir}}(H,R_1)$ with respect to $R_1$ and $H$ yields
	${\partial \tau_{\mathrm{nadir}}}/{\partial R_1} < 0$ and
	${\partial \tau_{\mathrm{nadir}}}/{\partial H} > 0$.
	Thus, $\tau_{\mathrm{nadir}}$ strictly decreases with $R_1$ and increases with $H$.
	By definition of $R^{\mathrm{hi}}$, $\tau(\overline{H},R^{\mathrm{hi}})=t_{\mathrm{DB2}}$.
	For any $R_1 \ge R^{\mathrm{hi}}$ with $H \in [\underline{H},\overline{H}]$, we have:
	\vspace{-1pt}
	\[
	\tau_{\mathrm{nadir}}(H,R_1)
	\ \le\ \tau_{\mathrm{nadir}}(\overline H,R_1)
	\ \le\ \tau_{\mathrm{nadir}}(\overline H,R_1^{\mathrm{hi}})
	=t_{\mathrm{DB2}}, \
	\]
	for any $R_1 \le R_1^{\mathrm{lo}}$ with $H \in [\underline{H},\overline{H}]$, we have:
	\[
	\tau_{\mathrm{nadir}}(H,R_1)
	\ \ge\ \tau_{\mathrm{nadir}}(\underline H,R_1)
	\ \ge\ \tau_{\mathrm{nadir}}(\underline H,R_1^{\mathrm{lo}})
	=t_{\mathrm{DB2}}.
	\]
	
	For $R_1\in(R_1^{\mathrm{lo}},R_1^{\mathrm{hi}})$, monotonicity in $H$ implies\vspace{-2pt}
	\[
	\tau_{\mathrm{nadir}}(H,R_1) \;\ge\;
	\tau_{\mathrm{nadir}}(\underline H,R_1) \;\ge\;
	\tau_{\mathrm{nadir}}(\underline H,R_1^{\mathrm{hi}}) = t_{\mathrm{DB2}}-\mu,
	\]
	which defines $\mu:=t_{\mathrm{DB2}}-\tau_{\mathrm{nadir}}(\underline H,R_1^{\mathrm{hi}})\ge 0$.
	To obtain the closed form of $\mu$, first solve
	$\tau_{\mathrm{nadir}}(\overline H,R_1^{\mathrm{hi}})=t_{\mathrm{DB2}}$
	for $R_1^{\mathrm{hi}}$:
	\begin{gather}
		R_1^{\mathrm{hi}}
		= D\,\Delta P_{\mathrm{dis}} / \big( 2\overline H (e^{\beta/\overline H} - 1) \big), \\
		\beta = D\,(\Delta t / 2),\label{eq:proof}
	\end{gather}
	where $\Delta t = t_{\mathrm{DB2}}-t_{\mathrm{DB1}}$.
	Substituting \eqref{eq:proof} into
	$\tau_{\mathrm{nadir}}(\underline H,R_1^{\mathrm{hi}})$:
	\begin{align}
		\mu = \frac{2}{D} \Big[ \beta - \underline H\,\ln\!\big(1 + \dfrac{\overline H (e^{\beta/\overline H} - 1)}{\underline H}
		\big) \Big].
	\end{align}
\end{proof}

To further evaluate the magnitude of $\mu$, we apply a second-order expansion
$e^a-1=a+a^2/2+O(a^3)$ and $\ln(1+x)=x-x^2/2+O(x^3)$, yielding
\vspace{-2pt}
\begin{align}
	\mu = \dfrac{\beta^2}{D} \frac{\overline H-\underline H}{\overline H\,\underline H} + O \Big(\dfrac{\beta^3}{D \overline H^2} \Big).
\end{align}
Normalizing by $\Delta t=2\beta/D$ gives
\vspace{-1pt}
\begin{align}
	\mu / \Delta t \approx \dfrac{D \Delta t} {4}\dfrac{\overline H-\underline H}{\overline H \underline H},
\end{align}
which typically leads to a relative error of only a few percent for moderate $\Delta t$ and $ H$,
and can thus be regarded as negligible in practice.

\section{FSPS-embeded DRCC-FC scheduling of the off-grid ReP2H system}
\label{sec:SchedulingModel}

\begin{figure}[t]
	\centering
	\includegraphics[width=5.7in]{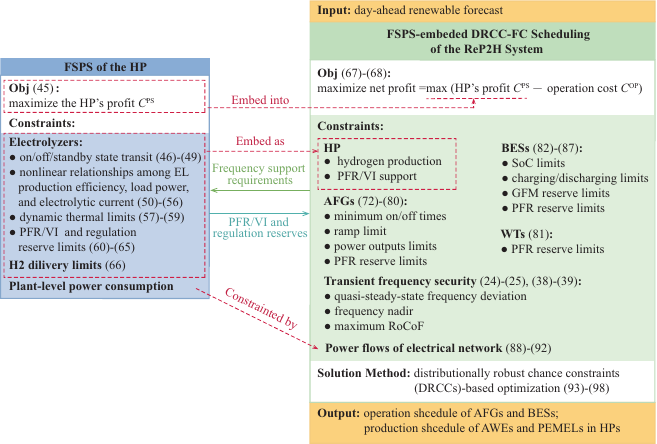}\vspace{-4.5pt}
	\caption{The framework of the FSPS-embeded DRCC-FC scheduling model of the off-grid ReP2H system.}
	\label{fig:OptiFramework}
\end{figure}
This section develops a frequency-supporting production scheduling (FSPS) model for the HP to ensure reliable frequency support without sacrificing hydrogen production profitability. The FSPS is then embedded into an FC scheduling framework for the off-grid ReP2H system based on the security constraints derived in Section~\ref{sec:FreqDynamic}. To address renewable uncertainty, distributionally robust chance constraints (DRCCs) are employed. The overall framework of the scheduling model is illustrated in Fig. \ref{fig:OptiFramework}. The resulting FSPS-embedded DRCC-FC approach jointly ensures economic efficiency and transient frequency security for the off-grid ReP2H system.

\subsection{FSPS formulation of the HP}
\emph{1) Objective function:}
the FSPS aims to maximize the HP's profit $C^{\mathrm{PS}}$, which is defined as hydrogen production revenues minus ELs' operating costs:
\begin{align}
	&C^{\mathrm{PS}}=
	\sum_{{t\in \mathcal{T}}}\sum_{{e\in \mathcal{E}}}
	\big(
	c_{e}^{\mathrm{H_{2}}} q_{e,t} \Delta t
	- c_{e}^{\mathrm{up}} z_{e,t}^{\mathrm{up,c}}
	- c_{e}^{\mathrm{down}} z_{e,t}^{\mathrm{down}}
	- c_{e}^{\mathrm{sb}} x_{e,t}^{\mathrm{sb}}
	\big),   \label{eq:66}
\end{align}
where $C^{\mathrm{PS}}$ is the FSPS objective representing the HP's profit, comprising ELs' operating costs and hydrogen production revenues;

\emph{2) Operational constraints of ELs:}
following our previous work~\cite{zhu2026exploring,zeng2024scheduling}, the start-up/standby/shutdown state transitions of an electrolyzer are described by \eqref{eq:67}--\eqref{eq:68}; the nonlinear relationships among EL production efficiency, load power, and electrolytic current are approximated by \eqref{eq:69}--\eqref{eq:72} using piecewise linearization and polynomial fitting. The auxiliary power are modeled as $P_{e,t}^{\mathrm{cool}} = \eta_{e}^{\mathrm{cool}} h_{e,t}^{\mathrm{cool}}$ and $P_{e,t}^{\mathrm{pump}} = K^\mathrm{pump}q_{e,t}^{\mathrm{H_2}}$; $\delta_{t}^{\mathrm{sb}}$ and $\delta_{t}^{\mathrm{sd}}$ are binary auxiliary variables.
\begin{gather}
	x_{e,t}^{\mathrm{st}} + x_{e,t}^{\mathrm{sb}} + x_{e,t}^{\mathrm{sd}} = 1, \\
	 z_{e,t}^{\mathrm{down}} = x_{e,t-1}^{\mathrm{st}} x_{e,t}^{\mathrm{sd}}, \label{eq:67}\\
	z_{e,t}^{\mathrm{up,h}} = x_{e,t-1}^{\mathrm{sb}} x_{e,t}^{\mathrm{st}}, \\
	 z_{e,t}^{\mathrm{up,c}} = x_{e,t-1}^{\mathrm{sd}} x_{e,t}^{\mathrm{st}}, \label{eq:68}\\
	q_{e,t}^{\mathrm{H_2}} \le \mathbf{A} I_{t} + \mathbf{B} x_{e,t}^{\mathrm{st}}, \\
      x_{e,t}^{\mathrm{st}} \underline{I} \le I_{t} \le x_{e,t}^{\mathrm{st}} \overline{I}, \label{eq:69}\\
	P_{e,t} = P_{e,t}^{\mathrm{stack}}(I,T) + P_{e,t}^{\mathrm{auxi}}, \\
	P_{e,t}^{\mathrm{auxi}} = x_{e,t}^{\mathrm{st}} P_{e,t}^{\mathrm{cool}} + P_{e,t}^{\mathrm{pump}}, \label{eq:70}\\
	P_{e,t}^{\mathrm{stack}}(I,T) = a_{e}^{1} I_{e,t} + a_{e}^{2} T_{e,t}
	+ x_{e,t}^{\mathrm{st}} a_{e}^{3} + \delta_{t}^{\mathrm{sb}} + \delta_{t}^{\mathrm{sd}}, \label{eq:71}\\
	-x_{e,t}^{\mathrm{sb}} M \le \delta_{t}^{\mathrm{sb}} \le x_{e,t}^{\mathrm{sb}} M, \\
	 -x_{e,t}^{\mathrm{sd}} M \le \delta_{t}^{\mathrm{sd}} \le x_{e,t}^{\mathrm{sd}} M. \label{eq:72} 
\end{gather}

The temperature dynamics of the electrolysis stack, which affect the P2H energy conversion efficiency, are modeled by~\eqref{eq:temp-1}--\eqref{eq:temp-2}, while the allowable operating temperature of the stack is constrained by~\eqref{eq:temp-3}.
\begin{gather}
	C_{e}^{\mathrm{heat}}
	(T_{e,t+1} - T_{e,t})
	= (P_{e,t}^{\mathrm{stack}} - n_{c} I_{e,t} V^{\mathrm{tn}}
	- h_{e,t}^{\mathrm{cool}}) \Delta t, \label{eq:temp-1} \\
	0 \le h_{e,t}^{\mathrm{cool}}
	\le a_{e}^{\mathrm{cool}} (T_{e,t} - T^{\mathrm{cool}}), \label{eq:temp-2} \\
      \underline{T}_e \le T_{e,t} \le \overline{T}_e \label{eq:temp-3}.
\end{gather}

To guarantee the HP delivers adequate frequency support, PEMELs and AWEs are required to satisfy the following PFR and VI reserve constraints, respectively:
\begin{gather}
	 P_{e,t} - R_{e,t}^{-} - R_{e,t}^{\mathrm{PFR}}
	\ge x_{e,t}^{\mathrm{sb}} P_{e,t}^{\mathrm{sb}}+ x_{e,t}^{\mathrm{st}} \underline{P}_{e}, \quad \forall e \in \mathcal{A}, \\
	 P_{e,t} + R_{e,t}^{+} + R_{e,t}^{\mathrm{PFR}}
	\le x_{e,t}^{\mathrm{sb}} P_{e,t}^{\mathrm{sb}}+ x_{e,t}^{\mathrm{st}} \overline{P}_{e}, \quad \forall e \in \mathcal{A},\\
 0 \le R_{e,t}^{\mathrm{PFR}} \le x_{e,t}^{\mathrm{st}} R_{e}^{\mathrm{PFR,lim}}, \quad \forall e \in \mathcal{A},\\
	 0 \le R_{e,t}^{-} \le x_{e,t}^{\mathrm{st}} R_{e}^{-,\mathrm{lim}}, \ 0 \le R_{e,t}^{+} \le x_{e,t}^{\mathrm{st}} R_{e}^{+,\mathrm{lim}}, \quad \forall e \in \mathcal{A},\\
	 P_{e,t} - R_{e,t}^{\mathrm{VI}}
	\ge x_{e,t}^{\mathrm{sb}} P_{e,t}^{\mathrm{sb}}+ x_{e,t}^{\mathrm{st}} \underline{P}_{e},\quad \forall e \in \mathcal{P}, \\
	 P_{e,t} + R_{e,t}^{\mathrm{VI}}
	\le x_{e,t}^{\mathrm{sb}} P_{e,t}^{\mathrm{sb}}+ x_{e,t}^{\mathrm{st}} \overline{P}_{e},\quad \forall e \in \mathcal{P},\label{eq:83}
\end{gather}
where the VI reserve are estimated as $R_{e,t}^{VI}=x_{e,t}^{\mathrm{st}}{{H}_{e}}\Delta f_{RoCoF}^{\mathrm{lim }}$.

\emph{3) Hydrogen delivery constraint:}
the hydrogen delivered by the HP to the downstream is constrained by the allowable throughput of the compressors. Thus the hydrogen production rate needs satisfy:
\begin{equation}
	\sum_{e\in \mathcal{E}} q_{e,t}^{\mathrm{H_2}}  \le \sum_{c\in \mathcal{C}} q_{c,t}^{\lim}. \label{eq:84}
\end{equation}

Note that more detailed pipeline and linepack constraints can also be imposed such as that in \cite{zeng2025planning}. Due to space limits and without loss of generality, they are not considered in this work.

\emph{4) Power consumption constraints:}
the power consumption of the HP is subject to system-level power balance constraints, which are detailed in Section~\ref{sec:power_flow}.

\subsection{FSPS-embedded DRCC-FC scheduling of the ReP2H system}

\emph{1) Objective function:} to embed the FSPS into the DRCC-based FC scheduling framework, the objective function of the system-level scheduling is set to maximize the net profit of the off-grid ReP2H system, as:
\begin{align}
	\max\; C^{\mathrm{net}} = C^{\mathrm{PS}} - C^{\mathrm{OP}}, \label{eq:85}
\end{align}
with
\begin{align}
	C^{\mathrm{OP}}=
	&\sum\limits_{t\in \mathcal{T}}\sum\limits_{g\in \mathcal G}
	\big(
	c_{g}^{\mathrm{NH_{3}}} q_{g,t}^{\mathrm{NH_{3}}}\Delta t
	+ c_{g}^{\mathrm{up}} z_{g,t}^{\mathrm{up}}  \big), \label{eq:42}
\end{align}
where $C^{\mathrm{OP}}$ is the system operation cost, including the fuel and start-up costs of AFGs.

\emph{2) Operational constraints of AFG:} the relationship  between power output and fuel consumption of AFGs is given by~\eqref{eq:47}, where $\eta_g^\mathrm{comb}$ and $\eta_g^\mathrm{steam}$ denote efficiency factors. AFG operational constraints include minimum on/off times~\eqref{eq:45}--\eqref{eq:46}, power outputs limits~\eqref{eq:48}--\eqref{eq:49}, PFR and regulation reserve limits~\eqref{eq:50}--\eqref{eq:51}, and ramping limits~\eqref{eq:52}--\eqref{eq:53}.
\begin{gather}
	P_{g,t}=\eta_g^\mathrm{comb}\eta_g^\mathrm{steam}LHV_\mathrm{NH_3}q_{g,t}^\mathrm{NH_3}, \label{eq:47}\\
	z_{g,t}^{\mathrm{up}}-z_{g,t}^{\mathrm{down}} = x_{g,t}^{\mathrm{st}}-x_{g,t-1}^\mathrm{st},\\
	z_{g,t}^{\mathrm{up}}+z_{g,t}^{\mathrm{down}} \leq 1,  \label{eq:43} \\
	T_{g}^{\mathrm{off}}(1-x_{g,t}^{\mathrm{st}})\leq\sum\nolimits_t^{t-1+T_{g}^{\mathrm{off}}}z_{g,t}^{\mathrm{down}}, \quad \forall t\leq T+1-T_{g}^{\mathrm{off}},  \label{eq:45}\\
	T_{g}^{\mathrm{on}}x_{g,t}^{\mathrm{st}} \leq\sum\nolimits_t^{t-1+T_{g}^{\mathrm{on}}}z_{g,t}^{\mathrm{up}}, \quad \forall t\leq T+1-T_{g}^{\mathrm{on}},  \label{eq:46}\\
	P_{g,t} - R_{g,t}^{-} - R_{g,t}^{\mathrm{PFR}} \ge x_{g,t}^{\mathrm{st}} \underline{P}_{g}, \label{eq:48}\\
	P_{g,t} + R_{g,t}^{+} + R_{g,t}^{\mathrm{PFR}} \le x_{g,t}^{\mathrm{st}} \overline{P}_{g}, \label{eq:49}\\
	0 \le R_{g,t}^{\mathrm{PFR}} \le x_{g,t}^{\mathrm{st}} R_{g}^{\mathrm{PFR,lim}}, \label{eq:50}\\
	0 \le R_{g,t}^{-} \le x_{g,t}^{\mathrm{st}} R_{g}^{-,\mathrm{lim}}, \\
	0 \le R_{g,t}^{+} \le x_{g,t}^{\mathrm{st}} R_{g}^{+,\mathrm{lim}}, \label{eq:51}\\
	P_{g,t} + R_{g,t}^{+} + R_{g,t}^{\mathrm{PFR}}
	- P_{g,t-1} - R_{g,t-1}^{-} - R_{g,t-1}^{\mathrm{PFR}}
	\le r_{g}^u, \label{eq:52}\\
	P_{g,t-1} - R_{g,t-1}^{-} - R_{g,t-1}^{\mathrm{PFR}}
	- P_{g,t} + R_{g,t}^{+} + R_{g,t}^{\mathrm{PFR}}
	\le r_{g}^d.\label{eq:53}
\end{gather}

\emph{3) Constraints of WTs:} the PFR provided by WTs needs to remain within the reserved margin under deloading operation. To account for WT output forecast uncertainties, this requirement is modeled as a distributionally robust chance constraint:
\begin{equation}
	\inf_{\mathbb{P} \in \mathbb{P}_{w}}
	\mathbb{P} \{ R_{w,t}^{\mathrm{PFR}}
	\le k_{w}^{\mathrm{deload}}\hat{P}_{w}+\xi_{w}^{'}  \}
	\ge 1 - \rho_{w}.
	\label{eq:86}
\end{equation}

\emph{4) Constraints of BESs:} the operation of BESs must satisfy the power and state-of-charge (SoC) constraints~\eqref{eq:55}--\eqref{eq:57}, where $R_{b,t}^{\mathrm{FR}} = R_{b,t}^{\mathrm{PFR}} + R_{b,t}^{\mathrm{GFM}}$, and the GFM reserve is estimated as $R_{b}^{\mathrm{GFM}} \approx \sum_{{b\in \mathcal{B}}} H_{b} \cdot \Delta f_{\mathrm{RoCoF}}^{\mathrm{lim}}$.
\begin{gather}
	0 \le P_{b,t}^{\mathrm C} + R_{b}^{+} + R_{b}^{\mathrm{PFR}} + R_{b}^{\mathrm{GFM}}
	\le x_{b,t}^{\mathrm C} \overline{P}_{b,t}, \label{eq:55}\\
	0 \le P_{b,t}^{\mathrm D} + R_{b}^{-} + R_{b}^{\mathrm{PFR}} + R_{b}^{\mathrm{GFM}}
	\le (1 - x_{b,t}^{\mathrm C}) \overline{P}_{b,t}, \label{eq:56}\\
	E_{b,t} = (1-\mu_{b}) E_{b,t-1} +
	( \eta_{b}^{\mathrm C} P_{b,t}^{\mathrm C}
	- P_{b,t}^{\mathrm D}/\eta_{b}^{\mathrm D} ) \Delta t, \label{eq:58}\\
	\frac{R_{b,t}^{\mathrm{FR}}+ R_{b,t}^{-}}{\eta_{b}^{\mathrm D}}
	+ \underline{E}_{b} \le E_{b,t}
	\le \frac{R_{b,t}^{\mathrm{FR}} + R_{b,t}^{+}}{\eta_{b}^{\mathrm C}}
	+ \overline{E}_b, \label{eq:60} \\
	R_{b}^{+} \ge 0, \ R_{b}^{\mathrm{PFR}} \ge 0, \\
	E_{b,0} = E_{b,\mathcal{T}}. \label{eq:57}
\end{gather}

\emph{5) Power flow constraints of the electrical network:}
\label{sec:power_flow}
ReP2H systems generally employ a radial topology \cite{zeng2024investment,zeng2024scheduling,zeng2025planning,zhu2026exploring}. Therefore, the power flow constraints in the scheduling problem can be approximated by a linearized DistFlow model~\cite{farivar2013branch}:
\begin{gather}
	P_{i,t}^{\mathrm{net}} + \sum\nolimits_{k \in \mathcal{F}(i)} P_{ki,t}
	= \sum\nolimits_{j \in \mathcal{\varsigma}(i)} P_{ij,t},
	\label{eq:pf-1} \\
	V_{i,t}^{2} - V_{j,t}^{2} = 2 R_{ij} P_{ij,t},
	\label{eq:pf-2} \\
	\underline{P}_{ij} \le P_{ij,t} \le \overline{P}_{ij},
	\label{eq:pf-3} \\
	\underline{V}_{i}^{2} \le V_{i,t}^{2} \le \overline{V}_{i}^{2}.
	\label{eq:pf-4}
\end{gather}
Here, \eqref{eq:pf-1} enforces nodal active power balance, and \eqref{eq:pf-2} relates squared voltage magnitude to branch active power flow. Constraints \eqref{eq:pf-3}--\eqref{eq:pf-4} impose branch flow limits and bus voltage bounds, respectively.
The net injected power at bus $i$ is modeled as:
\begin{align}
	P_{i,t}^{\mathrm{net}}
	=& \sum\nolimits_{g \in \mathcal{G}(i)} P_{g,t}
	+ \sum\nolimits_{w \in \mathcal{W}(i)} \bigl(1 - k_{w,t}^{\mathrm{deload}}\bigr)\,\hat{P}_{w,t}
	+ \hat{P}_{s,i,t} \nonumber\\
	&+ \sum\nolimits_{b \in \mathcal{B}(i)} P_{b,t}^{\mathrm{D}}
	- \sum\nolimits_{b \in \mathcal{B}(i)} P_{b,t}^{\mathrm{C}}
	- P_{\mathrm{HP},i,t}
	- P_{d,i,t},
	\label{eq:62}
\end{align}
where the power consumption of the HP is given by
$
P_{\mathrm{HP},i,t}
= \sum\nolimits_{e \in \mathcal{E}(i)} P_{e,t}
+ \sum\nolimits_{c \in \mathcal{C}(i)} P_{c,t}^{\mathrm{comp}} .
$

To address uncertainties of renewable outputs, we employ the electrolyzers, BESs, and AFGs to deploy upward and downward reserves to mitigate renewable forecast errors under an affine policy, similar to that in our previous work \cite{zhu2026exploring}. This adjustment is formulated as a distributionally robust chance constraint:
\begin{align}
	\inf_{\mathbb{P} \in \mathbb{P}_{r}}
	\mathbb{P} \Big\{
	&(\boldsymbol{\alpha}_{e,t}^{\mathrm{T}}
	+\boldsymbol{\alpha}_{g,t}^{\mathrm{T}}
	+\boldsymbol{\alpha}_{b,t}^{\mathrm{T}})\boldsymbol{\xi}_{r}
	-(R_{g,t}^{+}+R_{e,t}^{+}+R_{b,t}^{+})
	\le 0, \notag\\
	&(\boldsymbol{\alpha}_{e,t}^{\mathrm{T}}
	+\boldsymbol{\alpha}_{g,t}^{\mathrm{T}}
	+\boldsymbol{\alpha}_{b,t}^{\mathrm{T}})\boldsymbol{\xi}_{r}
	-(R_{g,t}^{-}+R_{e,t}^{-}+R_{b,t}^{-})
	\le 0
	\Big\}  \ge  1-\rho_{r}, \label{eq:87}
\end{align}
where $\boldsymbol{\xi}_{r} = \{\xi_{w}^{'}, \xi_{s}\}$; $\boldsymbol{\alpha}_{g,t} = \{\alpha_{g,t}^{w}, \alpha_{g,t}^{s}\}$, $\boldsymbol{\alpha}_{e,t} = \{\alpha_{e,t}^{w}, \alpha_{e,t}^{s}\}$, and $\boldsymbol{\alpha}_{b,t} = \{\alpha_{b,t}^{w}, \alpha_{b,t}^{s}\}$ are adjustment coefficients satisfying $\mathbf{-1} \le \{\boldsymbol{\alpha}_{g,t}^{\mathrm{T}}, \boldsymbol{\alpha}_{e,t}^{\mathrm{T}}, \boldsymbol{\alpha}_{b,t}^{\mathrm{T}}\}^{\mathrm{T}} \le \mathbf{1}$. To ensure power balance under uncertainties, the affine coefficients must satisfy $\boldsymbol{\alpha}_{g,t} + \boldsymbol{\alpha}_{e,t} + \boldsymbol{\alpha}_{b,t} = \mathbf{1}$. Moreover, regulation reserves of electrolyzers, AFGs, and BESs are also incorporated into the power flow constraints to guarantee feasibility during affine adjustments.

\emph{6) Frequency securitys constraints:}
the frequency security of the ReP2H system is enforced by incorporating constraints \eqref{eq:23}, \eqref{eq:24}, \eqref{eq:36}, and \eqref{eq:37} into the scheduling program. For the bilinear terms $x^{\mathrm{st}}R^{\mathrm{PFR}}$ in $\varphi_{1}(x)$ and $\varphi_{2}(x)$ of \eqref{eq:36} and \eqref{eq:37}, the big-M method can be applied to transform them into a mixed-integer linear form~\cite{bemporad1999control}.

\emph{7) Operational Constraints of HPs:}
the operational constraints of HPs in \eqref{eq:67}--\eqref{eq:84} are embedded into the FSPS, thereby enabling the HPs to provide adequate frequency support while maintaining profitability from hydrogen production.

\subsection{Solving method}
\label{sec:DRCC-CO}

Given the stochastic nature of renewable forecast errors and the limited knowledge of their underlying probability distributions, a distributional uncertainty model is adopted to address the DRCCs \eqref{eq:86} and \eqref{eq:87}. We first consider the distributions of WTs and PV output forecast errors $\mathbb{P}_{w}$ and $\mathbb{P}_{r}$ lie within ambiguity sets centered at their empirical distributions. Then, we construct these ambiguity sets in a data-driven manner from historical data and use Wasserstein distance to quantify the distance between the underlying and empirical distributions. Accordingly, $\mathbb{P}_{w}$ and $\mathbb{P}_{r}$ are defined as:
\begin{gather}
	\mathbb{P}_{w} := \{ \mathbb{P}_{w} \in \mathbb{M}_{w} : W_{p}(\mathbb{P}_{w}, \hat{\mathbb{P}}_{w}) \le \theta_w \}, \label{eq:Set-1}\\
	\mathbb{P}_{r} := \{ \mathbb{P}_{r} \in \mathbb{M}_{r} : W_{p}(\mathbb{P}_{r}, \hat{\mathbb{P}}_{r}) \le \theta_r \},
	\label{eq:Set-2}
\end{gather}
where $\hat{\mathbb{P}}_{w}$ and $\hat{\mathbb{P}}_{r}$ are the empirical distributions of WTs and PV output forecast errors. Given historical samples $\{ \hat{\xi}_{w,i}^{'} \}_{i=1}^{N_{w}}$ and $\{ \hat{\boldsymbol{\xi}}_{r,i} \}_{i=1}^{N_{r}}$, the empirical distributions are $\hat{\mathbb{P}}_{w} = \frac{1}{N_{w}} \sum_{i=1}^{N_{w}} \tilde{\xi}_{w,i}^{'}$ and
$\hat{\mathbb{P}}_{r} = \frac{1}{N_{r}} \sum_{i=1}^{N_{r}} \tilde{\boldsymbol{\xi}}_{r,i}$,
where $\tilde{\xi}_{w,i}^{'}$ and $\tilde{\boldsymbol{\xi}}_{r,i}$ are Dirac measures of the samples $\hat{\xi}_{w,i}^{'}$ and $\hat{\boldsymbol{\xi}}_{r,i}$. The Wasserstein metric $W_{p}(\mathbb{P}, \hat{\mathbb{P}})$ quantifies the discrepancy between the empirical and the underlying distributions, while the radius $\theta$ determines the conservativeness of the DR formulation.
Detailed computational procedures of Wasserstein metric can be found in~\cite{duan2018distributionally}.

Following \cite{duan2018distributionally,zymler2013distributionally}, when $\rho_{w} \le \frac{1}{N_{w}}$ and $\rho_{r} \le \frac{1}{N_{r}}$, the DRCCs in \eqref{eq:86} and \eqref{eq:87} can be equivalently reformulated under the ambiguity sets defined in \eqref{eq:Set-1} and \eqref{eq:Set-2} as:
\begin{gather}
	\beta + \frac{v\theta + \tfrac{1}{N} \sum_{i=1}^{N} k_{i}}{\rho} \le 0, \label{eq:89} \\
	\boldsymbol{A}(\boldsymbol{x})\,\hat{\boldsymbol{\zeta}}_{i} - B(x) - \beta \le k_{i}, \\
	\| \boldsymbol{A}(\boldsymbol{x}) \|_{\infty} \le v, \ v \ge 0, \ k_{i} \ge 0.\label{eq:90}
\end{gather}
Here, $\beta$, $k$, and $v$ are auxiliaries.
For constraint~\eqref{eq:86}, we have
$\boldsymbol{A}(\boldsymbol{x}) = -k_{\mathrm{deload}}$,
$\hat{\boldsymbol{\zeta}}_{i} = \hat{\xi}_{w,i}^{'}$,
$B(x) = k_{\mathrm{deload}}\hat{P}_{w,t} - R_{w,t}^{\mathrm{PFR}}$,
$\rho = \rho_{w}$, and $N = N_{w}$.
For constraint~\eqref{eq:87}, we have
$\boldsymbol{A}(\boldsymbol{x}) = \boldsymbol{\alpha}_{g,t} + \boldsymbol{\alpha}_{b,t} + \boldsymbol{\alpha}_{e,t}$,
$\hat{\boldsymbol{\zeta}}_{i} = \hat{\boldsymbol{\xi}}_{r,i}$,
$B(x) = R_{g,t}^s + R_{b,t}^s + R_{e,t}^s$, $s \in \{ +, - \}$,
$\rho = \rho_{r}$, and $N = N_{r}$.
Thus, the FSPS-embedded DRCC-FC scheduling model is transformed into a mixed-integer linear programming (MILP) problem solvable by commercial solvers.

\section{Case studies}
\label{sec:CaseStudy}

Two test systems are used to validate the proposed method. The first is based on a ReP2H demonstration project in Inner Mongolia, China~\cite{Songyuan2023}, as shown in Fig.~\ref{fig:diagram-1}. It includes eight $6.25~\mathrm{MW}$ WTs, a $10~\mathrm{MW}$ PV plant, a $40~\mathrm{MW}$ HP, an $8~\mathrm{MW}/8~\mathrm{MWh}$ BES, and three $12~\mathrm{MW}$ AFGs. The second is a large-scale ReP2H system constructed on a modified IEEE 69-bus network (Fig.~\ref{fig:diagram-2}) to verify scalability and general applicability. It contains twelve $6.25~\mathrm{MW}$ WTs, a $25~\mathrm{MW}$ PV plant, two $40~\mathrm{MW}$ HPs, a $20~\mathrm{MW}/20~\mathrm{MWh}$ BES, and five $6~\mathrm{MW}$ AFGs. Each HP consists of six $5~\mathrm{MW}$ AWEs and eight $1.25~\mathrm{MW}$ PEMELs, with parameters adopted from~\cite{cheng2025power,zeng2024scheduling}.

\begin{figure}[t]
	\centering
	\includegraphics[width=3.95in]{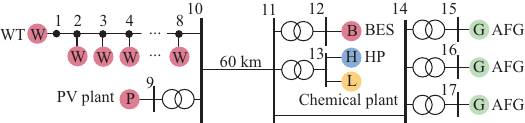}\vspace{-4.5pt}
	\caption{Schematic of the base test ReP2H system.}
	\label{fig:diagram-1}\vspace{0pt}
\end{figure}

\begin{figure}[t]
	\centering
	\includegraphics[width=6.2in]{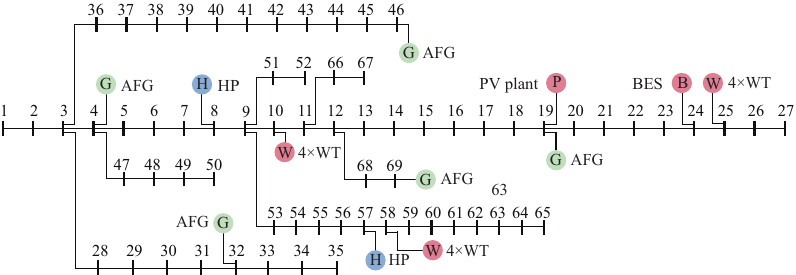}\vspace{-4.5pt}
	\caption{Schematic of the large-scale test ReP2H system based on the modified IEEE 69-bus system.}
	\label{fig:diagram-2}
\end{figure}

A $15\%$ step increase in total load is considered as the contingency that triggers a frequency deviation in the test systems. Frequency limits are set as
$f_N=50,\mathrm{Hz}$,
$\Delta f_\mathrm{nadir}^{\mathrm{lim}}=1.0,\mathrm{Hz}$,
$\Delta f_\mathrm{RoCoF}^{\mathrm{lim}}=0.5,\mathrm{Hz/s}$, and
$\Delta f_\mathrm{qss}^{\mathrm{lim}}=0.5,\mathrm{Hz}$~\cite{IEC62898_2023,farrokhabadi2019microgrid}.
PFR delivery times are ${t_b,t_e,t_w,t_g}={2,3,4,6},\mathrm{s}$.
The violation probability of chance constraints is set as $0.05$.
Forecast uncertainty of renewables is represented by $N_w=N_r=500$ real-world historical samples from~\cite{zhu2026exploring}.
Detailed parameters are listed in Tables~\ref{tab:para-sim} and \ref{tab:para-elec}.

Two comparison methods (CMs) are used as benchmarks to demonstrate the advantages of the proposed method (PM):

\textbf{CM1}: The general hydrogen-based system scheduling method \cite{qiu2023extend} without frequency security constraints.

\textbf{CM2}: The FC scheduling model in \cite{chu2021frequency,chu2024scheduling}, where frequency is regulated by AFGs, WTs, and BESs. HPs do not paricipate in frequency regulation.

\textbf{PM}: The proposed FSPS-embedded FC scheduling method with frequency support from AFGs, WTs, BESs, and HPs.

For a fair comparison, CM1 and the CM2 uses the same production scheduling formulations for the HP as the PM. All scheduling models are implemented in \emph{Wolfram Mathematica 14.0} and solved using \emph{Gurobi 12.0}. The scheduling step is $1\,\mathrm{h}$, and computations are carried out on a workstation with an \emph{Intel Core i9-13900K @ 3.0\,GHz} CPU and \emph{32\,GB RAM}.

\begin{table}[t]
	\centering
	\caption{Parameters of AFGs, BESs, WTs, and unit costs used in the case study.}
	\label{tab:para-sim}
	\vspace{-6pt}
	\footnotesize
	\renewcommand{\arraystretch}{1.05}
	\setlength{\tabcolsep}{6pt} 
	\begin{tabular}{ll|ll}
		\hline \hline
		Parameter & Value & Parameter & Value \\
		\hline
		$\eta_g^{\mathrm{comb}} / \eta_g^{\mathrm{steam}}$  & 0.88 / 0.40           & $\eta_{b}^{\mathrm{C}} / \eta_{b}^{\mathrm{D}}$ & 0.9 / 0.95 \\
		$T_{g}^{\mathrm{on/off}}$                           & 3 / 3 h               & $H_b$ & 1 MW$\cdot$s$^2$ \\
		$\underline{P}_{g} / \overline{P}_{g}$              & 4.5 / 12; 2 / 6 MW    & $\underline{V}_{i} / \overline{V}_{i}$   &  0.95 / 1.05 p.u.   \\
		$R_{g}^{\mathrm{PFR,lim}}$                          & $0.25 \, \overline{P}_{g}$ &  $c_{g}^{\mathrm{NH_{3}}}$  & 5 CNY/kg   \\
		$R_{g}^{+,\mathrm{lim}} / R_{g}^{-,\mathrm{lim}}$   & $0.05 \, \overline{P}_{g}$  & $c_{g}^{\mathrm{up}}$ & 1250 CNY \\
		$H_g$                                               & 3 s                   & $c_{e}^{\mathrm{H_{2}}}$     &32.9 CNY/MWh \\
		$k_{w}^{\mathrm{deload}}$                           & 0.1                   &  $c_{e}^{\mathrm{up/down}}$  &  800 / 0 CNY  \\
		$E_{b}^{\min} / E_{b}^{\max}$                       & 0.1 / 0.9 $\overline{P}_{b}$ &       &           \\
		\hline \hline
	\end{tabular}
\end{table}

\begin{table}[t]
	\centering
	\caption{Parameters of AWEs and PEMELs used in the case study.}
	\label{tab:para-elec}
	\vspace{-6pt}
	\footnotesize
	\renewcommand{\arraystretch}{1.05}
	\setlength{\tabcolsep}{6pt} 
	\begin{tabular}{ll|ll}
		\hline \hline
	 Parameter &  Value (AWE) & Parameter & Value (PEMEL) \\
		\hline
		$C_{e}^{\mathrm{heat}}$ & $7.8 \times 10^{7}$ J/$^\circ$C &
		$C_{e}^{\mathrm{heat}}$ & $2.0 \times 10^{7}$ J/$^\circ$C \\
		$a_{e}^{\mathrm{cool}}$ & 17 kW/$^\circ$C &
		$a_{e}^{\mathrm{cool}}$ & 17 kW/$^\circ$C \\
		$n_{c}/A$ & 313 / 4 m$^{2}$ &
		$n_{c}/A$ & 273 / 1 m$^{2}$ \\
		$\underline{T}/\overline{T}$ & 25 / 80 $^\circ$C &
		$\underline{T}/\overline{T}$ & 25 / 80 $^\circ$C \\
		$T^{\mathrm{cool}}/\eta_{e}^{\mathrm{cool}}$ & 5 $^\circ$C / 4 &
		$T^{\mathrm{cool}}/\eta_{e}^{\mathrm{cool}}$ & 5 $^\circ$C / 4 \\
		$\underline{I}/\overline{I}$ & 2.30 / 7.99 kA &
		$\underline{I}/\overline{I}$ & 0.55 / 2.29 kA \\
		$V_{tn}$ & 1.23 V &  $V_{tn}$ & 1.23 V \\
		\hline \hline
	\end{tabular}
\end{table}


\subsection{Base test system with 14 ELs and 3 AFGs}

\begin{figure}[t]
	\centering
	\includegraphics[width=4in]{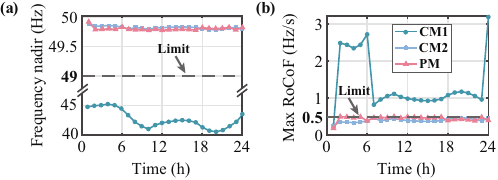}\vspace{-4.5pt}
	\caption{Hourly post-contingency frequency metrics of the base test system under the CM1, CM2, and the PM. (a) Frequency nadir. (b) Maximum RoCoF.}
	\label{fig:FrequencyMetric_1}
\end{figure}

\begin{figure}[t]
	\centering
	\includegraphics[width=3.85in]{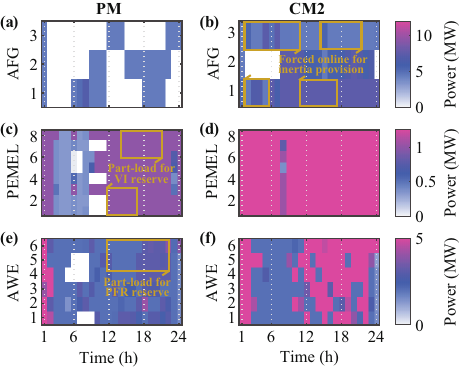}\vspace{-4.5pt}
	\caption{Scheduling results of AFGs, AWEs, and PEMELs of the base test system under PM and CM2.}
	\label{fig:SchedulingResults_1}
\end{figure}

\subsubsection{Necessity of frequency security constraints}

Fig.~\ref{fig:FrequencyMetric_1} shows the hourly frequency nadir and maximum RoCoF metrics under the assumed $15\%$ load disturbance over the scheduling day. Without transient frequency constraints, CM1 fails to secure frequency stability, with nadirs dropping below $45~\mathrm{Hz}$ and as low as $40.44~\mathrm{Hz}$. In contrast, CM2 and PM explicitly enforce nadir, RoCoF, and quasi-steady-state deviation constraints on inertia and PFR reserves, keeping the nadir and RoCoF within the $49~\mathrm{Hz}$ and $0.5~\mathrm{Hz/s}$ limits. These results demonstrate the necessity of embedding the frequency security constraints into scheduling.

\subsubsection{Impact of frequency support from HPs}

\emph{1) Impact on system operation:}
with VI support from PEMELs, HP participation alters the on-off combination of AFGs. During periods when AFGs would otherwise stay online solely to maintain inertia (e.g., periods 2--5 and 12--21 in Fig.~\ref{fig:SchedulingResults_1}), PEMELs supply the required inertia, enabling AFGs 1 and 3 to shut down under the PM as shown in Fig.~\ref{fig:SchedulingResults_1}(a). In contrast, under the CM2, where the HP cannot provide inertia support, AFGs are forced to remain online during these periods as shown in Fig.~\ref{fig:SchedulingResults_1}(b).

\emph{2) Impact on hydrogen production:}
providing PFR and VI requires electrolyzers to maintain load headroom, so both PEMELs and AWEs operate at part load under the PM, as shown in Figs.~\ref{fig:SchedulingResults_1}(c) and \ref{fig:SchedulingResults_1}(e). In contrast, under the CM2, AFGs remaining online for inertia support force electrolyzers to absorb surplus power and operate near full load unnecessarily, resulting in $35.32\%$ excessive hydrogen yield (uneconomically powered by AFGs, however) compared with the PM.

\emph{3) Impact on PFR allocation and economic performance:}
Figs.~\ref{fig:PFRResAndProfit}(a)--(c) show that under PM, AWEs provide $96.85\%$ of AFGs' PFR reserves, taking over most frequency support duties, which is another reason for allowing AFGs 1 and 3 to stay offline. Additionally, $55.52\%$ of WTs' PFR reserve demand is met by AWEs. Since WTs provide PFR through deloading, HP participation allows reducing their deloading rate (e.g., from $10\%$ to $5\%$), increasing wind energy utilization, and potentially increasing hydrogen production and system-level revenues.

Fig.~\ref{fig:PFRResAndProfit}(d) compares system operation costs $C^{\mathrm{UC}}$, net profit $C^{\mathrm{net}}$, and HP profit $C^{\mathrm{PS}}$. Under the PM, leveraging HP support reduces AFG runtime and costs, cutting total operation cost to $405,400$~CNY, compared with $1,160,000$~CNY under the CM2. Although the CM2 achieves higher HP profit, its excessive generation cost leads to a net loss of $424,400$~CNY. In contrast, the PM yields a net profit of $100,800$~CNY, a $123.74\%$ improvement in net profit over the CM2.

The above analysis shows the PM is economically superior. A more detailed discussion of the influence of HP's frequency support on its economic performance and system operation cost are given in Section~\ref{sec:EconomyAnalysis}.

\begin{figure}[t]
	\centering
	\includegraphics[width=4.1in]{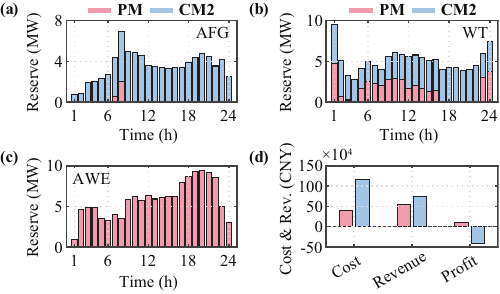}\vspace{-4.5pt}
	\caption{PFR reserves and economic performance of the base test system. (a) Total PFR reserve of AFGs. (b) Total PFR reserve of WTs. (c) Total PFR reserve of AWEs. (d) Operation cost $C^{\mathrm{OP}}$, HP revenue $C^{\mathrm{PS}}$, and net profit $C^{\mathrm{net}}$.}
	\label{fig:PFRResAndProfit}
\end{figure}

\subsubsection{Advantages of embedding the FSPS into FC scheduling:}

\begin{figure}[t]
	\centering
	\includegraphics[width=3.9in]{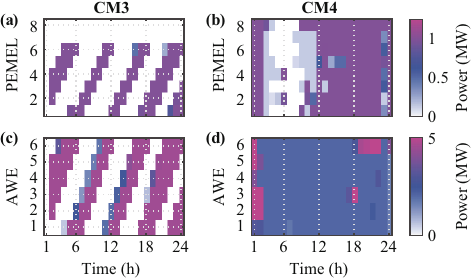}\vspace{-4.5pt}
	\caption{Scheduling results of AWEs and PEMELs of the base system under the CM3 and CM4.}
	\label{fig:OnOffStates_2}\vspace{0pt}
\end{figure}

\begin{figure}[t]
	\centering
	\includegraphics[width=4.in]{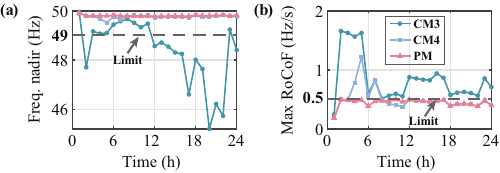}\vspace{-4.5pt}
	\caption{Hourly post-contingency frequency metrics of the base test system under the CM3, CM4, and the PM. (a) Frequency nadir. (b) Maximum RoCoF.}
	\label{fig:FrequencyMetirc_2}\vspace{0pt}
\end{figure}

HP participation in system frequency regulation involves three steps:

\emph{Step~1}: Determining the inertia and PFR required from the HP based on the FC scheduling of the ReP2H system;

\emph{Step~2}: Conduct HP production scheduling;

\emph{Step~3}: Allocating the reserves from \emph{Step~1} among each electrolyzers.

However, existing studies treated FC scheduling of systems and HP production scheduling separately, without exploring the allocation of frequency reserves among multiple electrolyzers. In contrast, the PM could simultaneously optimize system operation, HP scheduling, and the allocation of frequency reserves among AFGs, BESs, WTs, and electrolyzers. To demonstrate its benefits, two more CMs based on state-of-the-art studies are evaluated.

\textbf{CM3}: \emph{Step~1} uses the HP-assisted FC scheduling method in~\cite{wu2022incentivizing,li2024cvar} with the HP modeled as an aggregate electrolyzer; \emph{Step~2} operates the HP by rotational rules~\cite{li2023exploration,zheng2023off}; \emph{Step~3} allocates PFR/VI reserves based on the equimarginal principle~\cite{zeng2024scheduling}.

\textbf{CM4}: Identical to CM3 for \emph{Steps~1 and 3}, but \emph{Step~2} schedule HP uses the optimization-based strategy from \cite{li2024two,zheng2022optimal}.

In both the CM3 and the CM4, the reserve allocation in \emph{Step~3} follows the equimarginal principle to maximize efficiency, ensuring a fair comparison with the PM. The unit combination schedules of AFGs in the CM3 and the CM4 are consistent with those in the PM shown in Fig.~\ref{fig:SchedulingResults_1}(a), and Fig.~\ref{fig:FrequencyMetirc_2} shows the simulated post-disturbance frequency nadir and maximum RoCoF.

In the CM3, the HP is scheduled following the rotational rule to balance stack aging across electrolyzers. 
However, this strategy activates fewer electrolyzers in each period, reducing the inertia and frequency support actually available from the HP. Consequently, the post-disturbance frequency nadir and RoCoF exceed safety limits in multiple periods.

In the CM4, the number of active electrolyzers is similar to the PM, but HP scheduling is driven solely by hydrogen production without considering frequency support needs. Active PEMELs are insufficient in periods~4,~5, and~7, leading to inadequate inertia support and RoCoF violations. The PM keeps all frequency metrics within safety limits across all periods, achieving a hydrogen yield only $2.32\%$ lower than the CM3, balancing HP efficiency with reliable frequency support.

These results clearly show that the VI and PFR support capabilities of the HP are determined by the operating states of individual electrolyzers. Modeling the HP as a single aggregated electrolyzer and scheduling the off-grid ReP2H system and hydrogen production separately results in a mismatch between the frequency support required by the system and the actual capability of the HP, thereby compromising system frequency security. Hence, an FSPS-embedded FC scheduling framework is essential to maximize hydrogen production profitability while ensuring the HP consistently provides the necessary frequency support.

\subsection{Large-scale test system with 28 ELs and 5 AFGs}
\label{sec:IEEE-69}

This section evaluates the scalability of the PM on the larger-scale off-grid ReP2H system shown in~Fig.~\ref{fig:diagram-2} and further analyzes economic benefits from HP's frequency support over year-long scenarios.

\subsubsection{Computational Time and Impact of Renewable Power Uncertainty}
As noted in \cite{duan2018distributionally}, increasing the number of forecast error samples $N_{w/r}$ reduces the Wasserstein radius $\theta$ and hence the conservativeness of the DRCC-FC scheduling, while the reformulation in \eqref{eq:89}--\eqref{eq:90} grows linearly with $N_{w/r}$. Since these constraints involve only continuous variables, the computational complexity scales polynomially, approximately $\mathcal{O}((n+N_{w/r})^3)$ in the linear relaxation. To assess scalability, scenarios with different dataset sizes of $N_{w/r}$ were tested. As shown in Table~\ref{tab:thetaPerformance}, even for $\theta=0.0131$, the average solution time over 10 runs is $1141.43$~s, which is acceptable for practice. Furthermore, a smaller $\theta$ reduces conservativeness and yield higher net profit.

\begin{table}[t]
	\caption{System net profit and computation time with different ambiguity radius $\theta$.}
	\label{tab:thetaPerformance}
	\vspace{-6pt}
	\centering
	\footnotesize
	\renewcommand{\arraystretch}{1.05}
	\setlength{\tabcolsep}{6pt}
	\begin{tabular}{ccc}
		\hline \hline
		Ambiguity radius  $ \theta$
		& Net Profit (CNY)
		& Average CPU Time (s) \\
		\hline
		0.0248 & 481,613 &  644.43 \\
		0.0171 & 492,300 &  735.39 \\
		0.0149 & 499,071 & 1010.50 \\
		0.0131 & 499,607 & 1141.34 \\
		\hline \hline
	\end{tabular}
	\vspace{0pt}
\end{table}

\subsubsection{System operation cost, HP revenue, and net profit}
\label{sec:EconomyAnalysis}

\begin{table}[tb]
	\renewcommand{\arraystretch}{1.05}
	\caption{Comparison of the statistical averages of daily hydrogen yield, ammonia fuel consumption, net hydrogen yield, and frequency security metrics of the large-scale test system over 365 daily scenarios.}
	\label{tab:CM2_PM_Physical}
	\vspace{-6pt}
	\centering
	\footnotesize
	\begin{tabular}{cccccc}
		\hline\hline
		\multirow{2}{*}{Method}
		& Hydrogen yield
		& Ammonia fuel consumption
		& Net hydrogen yield
		& Frequency nadir
		& Max RoCoF \\
		& (t/day)
		& (t/day)
		& (t/day)
		& (Hz)
		& (Hz/s) \\
		\hline
		\textbf{CM2} & 25.58 & 85.56 & 10.52 & 49.72 & 0.397 \\
		\textbf{PM}  & 19.95 & \textbf{25.53} & \textbf{15.45} & 49.71 & 0.433 \\
		\hline\hline
	\end{tabular}

	\begin{tablenotes}[flushleft]
		\footnotesize
		\item {\scriptsize
			* Net hydrogen yield is calculated by converting ammonia fuel consumption into its hydrogen equivalent based on the stoichiometric reaction $\mathrm{N_2} + 3\mathrm{H_2} \rightarrow 2\mathrm{NH_3}$, where $1\,\mathrm{kg}$ NH$_3$ corresponds to approximately $0.176\,\mathrm{kg}$ H$_2$.}
	\end{tablenotes}
\end{table}

\begin{figure}[t]
	\centering
	\includegraphics[width=4.15in]{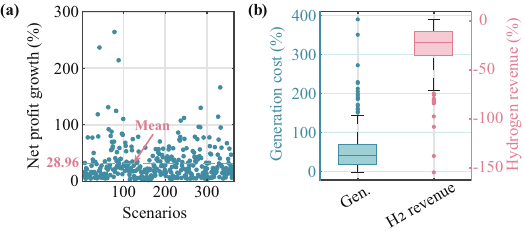}\vspace{-4.5pt}
	\caption{Relative improvement of the system net profit of the large-scale test system across 365 daily scenarios. (a) Relative improvement of net profit. (b) Contribution of generation and hydrogen production revenue to the net profit improvement.}
	\label{fig:improvRate}
\end{figure}

\begin{figure}[t]
	\centering
	\includegraphics[width=4in]{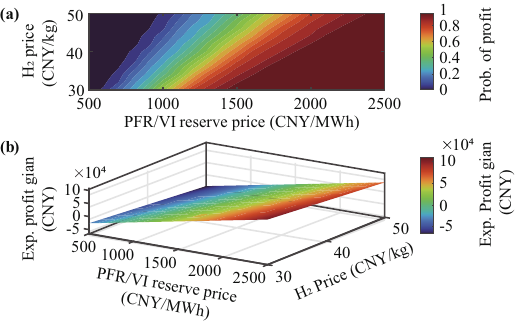}\vspace{-4.5pt}
	\caption{HP profit improvement probability and expected gain of the PM compared with the CM2 under the large-scale test system.}
	\label{fig:RevenueProb}
\end{figure}

Using 365 real daily operation scenarios from~\cite{zhu2026exploring}, Table~\ref{tab:CM2_PM_Physical} reports the statistical averages of daily hydrogen yield, ammonia fuel consumption of AFGs, and net hydrogen yield (considering the consumption for production ammonia fuel that is used to feed the AFGs) of the ReP2H system under the CM2 and PM. From the system perspective, although maintaining headroom for frequency support reduces hydrogen yield in the PM, the substitution of AFGs by HPs for PFR and inertia support lowers ammonia fuel consumption by $70.27\%$ compared with CM2. After converting ammonia savings into hydrogen equivalents, the PM delivers a higher net hydrogen yield and improves the system's net profit by an average of $28.96\%$ over CM2, as shown in Fig.~\ref{fig:improvRate}(a). Fig.~\ref{fig:improvRate}(b) further shows that $54.57\%$ of the profit gain comes from reduced generation costs, while lower hydrogen production offsets $26.08\%$ of the gain.

While the above analysis is performed from the view of a single economic entity at the system level, it is also of practical interest to assess the economic viability of the HP as an individual stakeholder~\cite{zeng2025planning}. In this case, the provision of PFR and VI is assumed to be remunerated, and the corresponding compensation is regarded as frequency regulation revenue for the HP.
Let $C_\mathrm{PM}^{\mathrm{H_2}}$ and $C_\mathrm{HP}^{\mathrm{Freq}}$ denote hydrogen and frequency regulation revenues under the PM, and $C_\mathrm{CM2}^{\mathrm{H_2}}$ denote hydrogen revenue under the CM2. The HP profit in the PM is:
$C_\mathrm{PM}^{\mathrm{HPRev}} = C_\mathrm{PM}^{\mathrm{H_2}} + C^{\mathrm{HPFreq}}.$
When $C_\mathrm{HP}^{\mathrm{Freq}} > C_\mathrm{CM2}^{\mathrm{H_2}} - C_\mathrm{PM}^{\mathrm{H_2}},$
the HP also achieves a net profit gain compared with the CM2.
Since $C_\mathrm{PM}^{\mathrm{H_2}}$, $C_\mathrm{HP}^{\mathrm{Freq}}$, and $C^{\mathrm{HPFreq}}$ depend on market prices, we evaluate profit improvement probability and expected profit gain of HP $(C_\mathrm{HP}^{\mathrm{Freq}} + C_\mathrm{PM}^{\mathrm{H_2}} - C_\mathrm{CM2}^{\mathrm{H_2}})$ across varying PFR/VI reserve and hydrogen prices. As shown in Fig.~\ref{fig:RevenueProb}, when the unit PFR/VI reserve price reaches $971\,\mathrm{CNY/MWh}$, expected profit turns positive, enabling HP arbitrage via frequency regulation. Beyond $1,500\,\mathrm{CNY/MWh}$, the probability of profit exceeds $50\%$, with expected gains ranging from $1,493.3$ to $37,033.1\,\mathrm{CNY}$, depending on the hydrogen price.

\section{Conclusion}
\label{sec:Conclusion}

This paper presents an FC scheduling framework for off-grid ReP2H systems, integrating an FSPS strategy for the HP. The proposed approach coordinates hydrogen production and frequency regulation, enabling AWEs and PEMELs to provide staged and reliable VI and PFR while preserving production profitability. Case studies based on a real-world ReP2H demonstration project and a large-scale system lead to the following conclusions:

1) A clear trade-off exists between frequency support capability and production efficiency in HPs. The VI and PFR capacities depend on the on/off operating states and loading levels of individual electrolyzers, which directly influence hydrogen output and efficiency.

2) Modeling the HP as an aggregated electrolyzer and separating FC scheduling from production scheduling can result in insufficient and unreliable frequency support. In contrast, the proposed integrated framework ensures consistent and adequate support while maintaining near-optimal hydrogen production.

3) Active participation of HPs in frequency regulation is recommended for off-grid ReP2H systems. Although reserving headroom for frequency support slightly reduces hydrogen output, HPs can effectively replace ammonia-fueled generators (AFGs) in providing PFR and inertia support. This substitution reduces AFG operating time and fuel consumption by $70.27\%$ and increases overall system net profit by $28.96\%$ compared with operation without HP-based frequency support, demonstrating the economic value of electrolyzer-based frequency regulation.

Future work will examine the impact of frequency regulation on electrolyzer degradation and determine the optimal level of HP participation in frequency control to balance degradation costs and system benefits.

\section*{Acknowledgement}

The authors gratefully acknowledge the financial support from the National Key Research and Development Program of China (2021YFB4000503) and the National Natural Science Foundation of China (52377116 and 52577129).

\section*{Declaration of Interest}

None.

\section*{Data Availability}

The data related to this work are available upon request.




\end{document}